\documentclass{amsart}
\usepackage{comment}
\usepackage{hyperref}
\usepackage{amssymb}
\usepackage{tikz-cd}
\usepackage[T1]{fontenc}
\usepackage[utf8]{inputenc}
\usepackage[english]{babel}

\usepackage{graphicx}%
\usepackage{multirow}%
\usepackage{amsmath,amssymb,amsfonts}%
\usepackage{amsthm}%
\usepackage{mathrsfs}%
\usepackage[title]{appendix}%
\usepackage{xcolor}%
\usepackage{textcomp}%
\usepackage{manyfoot}%
\usepackage{booktabs}%
\usepackage{algorithm}%
\usepackage{algorithmicx}%
\usepackage{algpseudocode}%
\usepackage{listings}%

\newcommand{\gi}{H^1_{0,g}(\Omega)}
\newcommand{\numberset}{\mathbb}
\newcommand{\N}{\numberset{N}}
\newcommand{\Z}{\numberset{Z}}
\newcommand{\R}{\numberset{R}}

\newcommand{\vp}{\varphi}
\newcommand{\oo}{\Omega}

\theoremstyle{plain} 
\newtheorem{thm}{Theorem}[section] 
\newtheorem{cor}[thm]{Corollary} 
\newtheorem{lem}[thm]{Lemma} 
\newtheorem{prop}[thm]{Proposition}

\theoremstyle{remark} 
\newtheorem{oss}[thm]{Remark}

\title[An isoperimetric inequality for twisted eigenvalues]{An isoperimetric inequality for twisted eigenvalues with one orthogonality constraint}
\author{Emanuele Salato}
\author{Davide Zucco}
\address[Emanuele Salato and Davide Zucco]{Dipartimento di Matematica ``G. Peano'', Università di Torino, Via Carlo Alberto 10, Torino, 10123, Italy}
\address[Emanuele Salato]{Dipartimento di Scienze Matematiche, Politecnico di Torino, Corso Duca degli Abruzzi 24, Torino, 10129, Italy.}
\address[Emanuele Salato]{Laboratoire de Mathématiques, Université Savoie Mont Blanc, UMR CNRS 5127, Campus Scientifique, Le Bourget-du-Lac, 73376, France.}
\email{emanuele.salato@unito.it}
\email{davide.zucco@unito.it}

\begin{document}

\begin{abstract}
We consider \emph{twisted} eigenvalues $\lambda_{1}^{g}(\oo)$, defined as the minimum of the Rayleigh quotient of functions in $H^1_{0}(\oo)$ that are
\emph{orthogonal} to a {given} function $g\in  L^2_\text{loc}(\mathbb R^d)$. 
We prove an \emph{isoperimetric inequality} for $\lambda_1^g(\Omega)$, which provides a uniform bound on twisted eigenvalues -- not only with respect to the  set  $\Omega$ (an open bounded set of $\mathbb R^d$) -- but also in relation to the orthogonality function $g$. 
Remarkably, the lower bound is \emph{uniquely} attained when $\Omega$ is the \emph{union of two disjoint balls} of specific radii, and when the function $g$ in the orthogonality constraint is of \emph{bang-bang} type, i.e., constant on each ball. As a consequence, we obtain a \emph{continuous} 1-parameter family of optimal sets -- each being the {union of two disjoint balls} -- that interpolates between the optimal shapes of the first two Dirichlet eigenvalues of the Laplacian.

This new isoperimetric inequality offers fresh perspectives on well-established results, such as the Hong-Krahn-Szeg{o} and the Freitas-Henrot inequalities.  Notably, only for these two particular inequalities our proof avoids reliance on Bessel functions, suggesting potential extensions to nonlinear settings. However, extending the inequalities to the general case requires proof strategies that rely on properties of Bessel functions.
\end{abstract}

\maketitle

{\small


}

\section{Introduction}

A wide variety of problems, arising from both pure and applied mathematics, involve the optimization of quantities with respect to functions subject to appropriate \emph{orthogonality constraints}. A typical situation occurs with the so-called \emph{min-max principles} (\cite{ch,sw}), which provide variational characterizations of the eigenvalues of compact self-adjoint operators on Hilbert spaces: the $k$-th eigenvalue of such operators can be obtained by minimizing the Rayleigh quotient among functions that are orthogonal  to the $(k-1)$ eigenfunctions $u_1, u_2, . . ., u_{k-1}$ corresponding to the first $k-1$ eigenvalues. Of course, to compute the $k$-th eigenvalue in this way, the $k-1$ eigenfunctions must be known beforehand. Another classical problem, originally considered by Lagrange in the 1770s (\cite{eg, L}), deals with the optimal shape of a column, where one optimizes with respect to functions having mean value zero (i.e., those that are orthogonal to constant functions in the $L^2$ sense). Similar orthogonality constraints have more recently appeared in problems related to constant mean curvature immersions \cite{bb}, to intermediate problems \cite{sw} as well as to hydrostatics \cite{f}, where, in the latter, the orthogonality condition is with respect to all harmonic functions.

Since the orthogonality constraint is generally unrelated to the ordinary Dirichlet boundary conditions, this phenomenon in \cite{bb} has been referred to as \textit{twisted} (in French, \emph{tordu}; see \cite{ber}). The majority of the literature on the subject considers either  orthogonality to \emph{unknown} functions -- not specified a priori, such as eigenfunctions -- or to \emph{highly specific} ones, such as constants or harmonic functions. 
Our goal is to develop a systematic framework for twisted questions, where optimization problems are subject to orthogonality constraints with respect to a \emph{given}, \emph{generic} family of functions. In this initial paper, we focus on the case of a \emph{single} orthogonality constraint and introduce tools for deriving isoperimetric inequalities. These inequalities provide uniform bounds on twisted eigenvalues -- not only with respect to the shape -- but also in relation to the orthogonality constraint, offering new perspectives on some of the previously discussed works.

More precisely, given the dimension $d \in \N$ with $d\ge 2$, an open bounded set $\oo \subset \R^d$, and a function $g$ in $L^2_\text{loc}(\R^d)$ (possibly but not necessarily depending on $\Omega$) we consider \textit{the space of twisted functions} (vanishing at the boundary)
    \begin{equation*}
\gi:=g^\perp\cap H^1_{0}(\oo)=
\left\{ u \in H^1_{0}(\oo) \, : \, \int_\oo u(x)g(x)  dx=0 \right\},
    \end{equation*}
 defined, roughly speaking, as the orthogonal complement of a function $g\in L^2_\text{loc}(\mathbb R^d)$, with respect to the scalar product of $L^2(\mathbb R^d)$, within the Sobolev space $H^1_0(\Omega)$, or equivalently as the space of $H^1_0(\Omega)$ functions which are orthogonal in $L^2(\mathbb R^d)$ to the function $g$ (as usual, functions in $H^1_0(\Omega)$ are extended by $0$ outside $\Omega$). This space does not change if the function $g$ is modified on a set of measure zero. Moreover, the values attained by $g$ in $\R^d \setminus \Omega$ do not affect the definition of the space. We choose $g \in L^2_\text{loc}(\mathbb R^d)$, instead of $g \in L^2(\Omega)$, since in the following we will need to perform scalings and translations.
By definition, $\gi$ is a non-empty Hilbert space with the scalar product endowed by $H^1_{0}(\Omega)$.
The \emph{(first) twisted eigenvalue of $\Omega$} is then defined as
    \begin{equation}\label{lambda}
        \lambda_{1}^{g}(\oo)
        :=\min_{u\in  H^1_{0,g}(\Omega)\atop u\not\equiv 0}
        \frac{\int_\oo |\nabla u(x)|^2  dx}{\int_\oo  u(x)^2  dx}.
    \end{equation}
The minimum problem is well-posed (there always exists a minimizer) and since for every $s \in \R \setminus \{0\}$ the space $H^1_{0,sg}(\Omega)=\gi$ with $(sg)(x)=s\cdot g(x)$ for every $x\in \mathbb R^d$, by \eqref{lambda} there holds
\begin{equation}\label{invariance2}
\lambda_1^{sg}(\Omega)=\lambda_1^{g}(\Omega),
\end{equation}
namely the twisted eigenvalue is invariant under the \emph{outer} scalings $sg$ of the function $g$. A function $u\in \gi$ achieving the minimum in \eqref{lambda} is called (first) \emph{twisted eigenfunction corresponding to $\lambda_{1}^{g}(\oo)$} and it solves in a weak sense
    \begin{equation}\label{pde}
    \left\{
\begin{array}{ll}
-\Delta u= \lambda_{1}^{g}(\oo) u +\bigg( \dfrac{\int_{\oo} -\Delta u(x)g(x) dx}{\int_\Omega g(x)^2 dx} \bigg)\, g\quad & \textup{in } \oo,
\\ 
u=0 & \textup{on } \partial\Omega.
\end{array}
\right.
    \end{equation}
when $g \not\equiv 0$, while it solves $-\Delta u= \lambda_{1}^{g}(\oo) u$ in $\Omega$ with $u=0$ on $\partial \Omega$, when $g \equiv 0$ (i.e. $u$ is a Dirichlet eigenfunction). Notice that, the preceding relations can be rewritten and unified as follows
    \begin{equation*}
    \left\{
        \begin{array}{ll}
(\text{Id}-P^g_\Omega)(-\Delta u)= \lambda_{1}^{g}(\oo) u\quad & \textup{in } \oo,
\\ 
u=0 & \textup{on } \partial\Omega,
\end{array}
\right.
    \end{equation*}
    where $\text{Id}$ and $P^g_\Omega$ denote the identity operator and the $L^2(\Omega)$ orthogonal projection onto the subspace $\text{span}(g|_\Omega)$, respectively.
    Due to the presence of an average of the Laplacian over $\oo$, equations like \eqref{pde} are referred to as \textit{non-local} (see, for instance, \cite[Introduction]{bhmp} and the references therein). The number of twisted eigenfunctions corresponding to the same twisted eigenvalue $\lambda_{1}^{g}(\oo)$ is called the \textit{multiplicity} of $\lambda_{1}^{g}(\oo)$. If the multiplicity is one, the eigenvalue is \emph{simple}; if it is two, the eigenvalue is \emph{double}, and so on. In general, $\lambda_{1}^{g}(\oo)$ can be \emph{multiple}, meaning that it is not simple (for instance, if $\oo$ is a disk in $\R^2$ and $g \in \{ 1, u_1 \}$, then $\lambda_{1}^{g}(\oo)$ is double, see \cite[Example 3]{bb}).
    
For specific choices of $g$, twisted eigenvalues \eqref{lambda} reduce to well-known quantities.
\begin{itemize}
        \item If $g\in \{ 0, \{u_k \}_{k \ge 2} \}$ with $u_k$ a \emph{Dirichlet eigenfunction corresponding to the $k$-th eigenvalue} of $\oo$, then 
        \begin{equation}\label{g.FK}
            \lambda_{1}^{g}(\oo)
            =\min_{{u\in  H^1_{0,g}(\Omega)\atop u\not\equiv 0}}\frac{\int_\oo |\nabla u(x)|^2  dx}{\int_\oo  u(x)^2  dx}
        =\lambda_{1}(\oo) ,
        \end{equation}
        where $\lambda_{1}(\oo)$ is the \emph{first Dirichlet eigenvalue} of $\oo$ (see \cite[Chapter~1, Theorem~1]{sw}). When $\Omega$ is connected it is always simple (see \cite[Theorem 6.34]{bor}).
        \item If $g=u_1$ with $u_1$ a \emph{Dirichlet eigenfunction corresponding to the first eigenvalue}  of $\oo$, then
        \begin{equation*}
            \lambda_{1}^{u_1}(\oo)
            =\min_{u\in H^1_{0,u_1}(\Omega)\atop u\not\equiv 0} \frac{\int_\oo |\nabla u(x)|^2  dx}{\int_\oo  u(x)^2  dx}
        =\lambda_{2}(\oo) ,
        \end{equation*}
        where $\lambda_{2}(\oo)$ is the \emph{second Dirichlet eigenvalue} of $\oo$ (see \cite[Chapter~1, Theorem~1]{sw}).
        \end{itemize}
        \begin{itemize}
        \item If $g\equiv1$ (i.e., if $g$ is the constant function identically equal to $1$), then
        \begin{equation}\label{g.t}
            \lambda_{1}^{1}(\oo)
            =\min_{{u\in H^1_{0,1}(\Omega)\atop u\not\equiv0}} \frac{\int_\oo |\nabla u(x)|^2  dx}{\int_\oo  u(x)^2  dx} 
        =\lambda_{1}^{T}(\oo) ,
        \end{equation}
        where $\lambda_{1}^{T}(\oo)$ is the \textit{first (standard) twisted eigenvalue} considered in \cite{bb,fh} (see in particular \cite[Proof of Proposition 2.2 (2)]{bb}).
    \end{itemize}
    
A natural question related to these eigenvalues, and one that is also relevant for the applications, is to determine \emph{uniform} bounds that are independent of both the function $g$ and the set $\Omega$. We start by focusing on the dependence on $g$ and notice that 
the first Dirichlet eigenvalue, can be characterized as the minimum of twisted eigenvalues 
 \begin{equation}\label{ming}
        \lambda_{1}(\oo)=\min_{g \in L^2_{\text{loc}}(\R^d)} \lambda_{1}^{g}(\oo),
    \end{equation}
while the second Dirichlet eigenvalue as the maximum of twisted eigenvalues
    \begin{equation}\label{maxg}
        \lambda_{2}(\oo)=\max_{g \in L^2_{\text{loc}}(\R^d)} \lambda_{1}^{g}(\oo),
    \end{equation}
    (see \cite[Chapter 3, Section 2]{sw}). In particular, for each $g\in L^2_\text{loc}(\mathbb R^d)$, twisted eigenvalues interlace between the first and second Dirichlet ones with
        \begin{equation}\label{interwinedTD}
            \lambda_{1}(\oo) 
            \le \lambda_{1}^{g}(\oo) 
            \le \lambda_{2}(\oo).
        \end{equation}
        
Two possible physical interpretations of these twisted eigenvalues are as follows. In Quantum Mechanics, the eigenvalues of the Dirichlet Laplacian correspond to the energy levels of a particle confined in $\Omega$ with infinite walls (i.e., an infinite potential barrier at the boundary), with the corresponding eigenfunctions representing the associated quantum states: $\lambda_1(\Omega)$ is the ground-state energy, while $\lambda_2(\Omega)$ is the first excited-state energy after the ground-state energy. In view of \eqref{interwinedTD} the eigenvalue $\lambda_1^g(\Omega)$ can be regarded as an intermediate energy where the orthogonality constraint required in \eqref{lambda} may be related to the so-called \emph{selection} and \emph{superselection rules} (see for instance the rule Ss1 and the relation (7.61) in \cite[Section 7.7]{mor}), which for various reasons (symmetries, external fields, etc...) certain transitions are forbidden and we know a priori that the wavefunction of the state must be orthogonal to some subspace: $\lambda_1^g(\Omega)$ may represent the minimum energy that such a state can achieve. Another interpretation can be found in \emph{stability theory} for symmetric solutions of dispersive equations (more specifically, standing waves and solitons for nonlinear Schrödinger equations, or even waves for Korteweg-de Vries-type equations), where one is interested in the study of the spectrum of differential operators on suitable Hilbert spaces that are constrained.  The constraint is typically expressed through an orthogonality with respect to a given set of functions and it is motivated by the fact that these functions form an element of the kernel, of another operator related to the same linearization of the energy, see for instance \cite{cm}. The theory was developed by M. Weinstein and independently by M. Grillakis, J. Shatah and W. Strauss in the late 1980s, and it was subsequently greatly expanded in the following years by numerous authors, with countless applications to linear and orbital stability theory, for more details see \cite[Section~5.2]{kp} (in particular equations (5.2.30) and (5.2.31)) and \cite[Section~4.2]{pel} (in particular equation (4.2.6)).

Now, returning to the mathematical discussion, notice that from \eqref{g.FK} there exist bounded minimizers of \eqref{ming}. The invariance under outer scalings \eqref{invariance2} allows us to rescale each bounded function in terms of its $\|\cdot\|_\infty$-norm, such that one may consider $|g(x)|\le 1$ for a.e. $x\in\mathbb R^d$. Consequently, the characterization \eqref{ming} still holds, with $L^2_{\text{loc}}(\mathbb R^d)$ replaced by the space of \emph{uniformly bounded} functions, see Proposition~\ref{p.sufficient}. After reviewing in Section~\ref{s.preliminary} some preliminary results on twisted eigenfunctions, we will prove in Section~\ref{sec.g}, see relation \eqref{infL+}, that the minimal value in \eqref{ming} can be achieved by working within the more restricted space of \emph{positive}, \emph{uniformly bounded} functions, which we denote by
$$L^\infty_+:=\{g\in L^\infty(\mathbb R^d):  0<g(x)\le 1 \text{ for a.e. } x\in \mathbb R^d\}.$$
Although the positivity constraint in the class above prevents the function $g$ from being identically zero and from being equal to any Dirichlet eigenfunction $u_k$ with $k\ge 2$, we still have
 \begin{equation}\label{infg}
        \lambda_{1}(\oo)=\inf_{g \in L^\infty_+} \lambda_{1}^{g}(\oo).
\end{equation}
Here the infimum is indeed not a minimum as it is no longer attained by any function in $L^\infty_+$. However, minimizing sequences do exists in $L^\infty_{+}$, converging to $\lambda_1(\Omega)$, and taking the form (the so-called \emph{bang-bang functions}, see below)
$$g_n:=\alpha_n\chi_{\Omega\setminus A_n} +\chi_{\R^d\setminus (\Omega\setminus A_n)},$$
for a suitable sequence of measurable sets $A_n$ and real numbers $\alpha_n$  such that  $|A_n|\to 0^+$ and $\alpha_n\to 0^+$ as $n\to\infty$ (see Propositions~\ref{p.necessary} and \ref{p.espilon} later on). 

To track the behaviour of minimizing sequences it is possible to restrict the problem within the smaller class of \emph{uniformly positive} and \emph{uniformly bounded} functions, defined for every $0<\alpha\le 1$ as the space
\begin{equation}\label{Lalpha}
L^\infty_\alpha:=\{g\in L^\infty(\mathbb R^d): \alpha \le g(x)\le 1 \text{ for a.e. $x\in \mathbb R^d$}\}.
\end{equation}
This class is closed under scalings. Other useful classes can be found in Remark~\ref{scalclass}.
A particular function in $L^\infty_\alpha$ is
 given by a \emph{bang-bang function}
$$\chi_\alpha=\alpha\chi_{\Omega^+} +\chi_{\mathbb R^d\setminus \Omega^+},$$
a function defined on two \emph{disjoint} regions (i.e., open bounded sets) $\Omega^+\cup \Omega^-$, which is $\alpha\in(0,1]$ in $\Omega^+$, 1 in $\Omega^-$ and extended by $1$ outside $\Omega^+\cup \Omega^-$. 
These functions will play a significant role throughout the paper and we use the notation $\chi_\alpha$ to denote them. These bang-bang functions are also of particular importance in Control Theory (see, for instance, \cite[Chapters 7, 8, and 9]{h} and the references therein).

As for the upper bound in \eqref{interwinedTD}, the situation is much simpler. When $\Omega$ is either a ball $B$ or the union of two disjoint balls $B^\wedge, B^\vee$ of equal measure, the maximization problem \eqref{maxg} is achieved by the constant function $g\equiv 1$. This follows from \eqref{g.t} with
\begin{equation*}
  \lambda_1^T(B)=\lambda_2(B)  \quad \text{and}\quad \lambda_{1}^{T}(B^\wedge \cup B^\vee)=\lambda_2(B^\wedge \cup B^\vee),
\end{equation*}
proved, respectively, in \cite[Example~3]{bb} and \cite[Corollary~2.3]{fh}.
Since $1\in L^\infty_\alpha$ for all $0<\alpha \le 1$, the second Dirichlet eigenvalue can be seen, in a sense, as a sharp upper bound for $\lambda_1^g(\Omega)$. Actually, in the case of two balls of equal measure, the twisted eigenvalue is independent of $g$. Indeed, when $\oo=B^\wedge \cup B^\vee$ then
$\lambda_1(B^\wedge\cup B^\vee)=\lambda_1(B^\wedge)=\lambda_2(B^\wedge\cup B^\vee)$ and
for every $g\in L^2_\text{loc}(\mathbb R^d)$, the bounds in \eqref{interwinedTD} yield
\begin{equation}\label{eq.ind}
\lambda_1^{g}(B^\wedge\cup B^\vee)=\lambda_1(B^\wedge\cup B^\vee)=\lambda_1(B^\wedge)=\lambda_2(B^\wedge\cup B^\vee).
\end{equation}

As for the shape dependence of $\lambda_1^g(\Omega)$ at a fixed $g$, let us move the set $\Omega$ in the class of bounded and open sets of $\mathbb R^d$, that we denote by $\mathcal O$. Twisted eigenvalues, similar to Dirichlet eigenvalues, possess valuable properties in dependence of the shape: they are \emph{monotone with respect to the set inclusion}
    \begin{equation*}
        \lambda_{1}^{g}(\oo_2) \le \lambda_{1}^{g}(\oo_1),\quad  \text{for all $\oo_1,\oo_2\in\mathcal O$  with $\oo_1 \subset \oo_2$},
    \end{equation*}
and they exhibit well-behaved \emph{scaling properties} 
    \begin{equation}\label{scaling}
        \lambda_{1}^{{g}}(s\oo)=\frac{\lambda_{1}^{g_s}(\oo)}{s^2}, \quad \text{for all $s\in\mathbb R\setminus\{0\}$ and $\oo\in\mathcal O$},
    \end{equation}
    where $g_s(x):=g(s\cdot x)$ for every $x\in \mathbb R^d$ are \emph{inner} scalings of the function $g$.
With the previous specific choices of $g$, let us recall the following well-known inequalities, where the cases of equality (the so-called \emph{rigidity})
should be understood up to sets of capacity zero (see, for instance, \cite[Section 3.3]{hp}).

\begin{itemize}
    \item The \emph{Faber-Krahn inequality} (see \cite[Theorem 3.2.1]{h}) states that
    \begin{equation}\label{FK}
        |\Omega|^\frac{2}{d}\lambda_1(\oo) \ge  |B|^\frac{2}{d}\lambda_1(B) \, ,
    \end{equation}
    where $B$ is any ball in $\mathbb R^d$. Equality holds if and only if $\oo=B$.
    \item The \emph{Hong-Krahn-Szeg{o} inequality}\footnote{This “isoperimetric” inequality was first proved by Edgar Krahn in the 1920s, but the result was later largely overlooked, since in 1955 George Pólya attributed the result to Peter Szego. However, almost in the same years as Pólya’s paper, a paper by Imsik Hong appeared, providing once again a proof of
    this result. It should be noted that Hong’s paper was published in 1954, just one year before Pólya’s. For this reason, we refer to \eqref{HKS} as the Hong-Krahn-Szego inequality. We thank Mark S. Ashbaugh for these historical information (further details, and references to these classical works, can be found in \cite{d} and \cite[p.6]{hp}).} (see \cite[Theorem 4.1.1]{h} or \cite[Theorem 6.4.1 and p.6]{hp}) states that
    \begin{equation}\label{HKS}
         |\Omega|^\frac{2}{d}\lambda_2(\oo) \ge  |B^\wedge \cup B^\vee|^\frac{2}{d}\lambda_2(B^\wedge \cup B^\vee) \, ,
    \end{equation}
    where $B^\wedge$ and $B^\vee$ are any disjoint balls in $\mathbb R^d$ of equal measure. Equality holds if and only if $\oo=B^\wedge \cup B^\vee$.
    \item The \emph{Freitas-Henrot inequality} (see \cite[Theorem 1]{fh}) states that
    \begin{equation}\label{FH}
        |\Omega|^\frac{2}{d} \lambda_{1}^{T}(\oo) \ge  |B^\wedge \cup B^\vee|^\frac{2}{d}\lambda_{1}^{T}(B^\wedge \cup B^\vee) \, ,
    \end{equation}
    where $B^\wedge$ and $B^\vee$ are any disjoint balls in $\mathbb R^d$ of equal measure. Equality holds if and only if $\oo=B^\wedge \cup B^\vee$.
\end{itemize}

According to the classical book \cite{posz51}, these inequalities are commonly referred to as \textit{isoperimetric}, even though no perimeter appears in them. We follow this convention also here in this paper. The presence of the measure inside \eqref{FK}, \eqref{HKS}, and \eqref{FH} ensures that the inequalities are scale-invariant. An equivalent formulation of these inequalities can be expressed as the minimization of eigenvalues over open sets subject to a \emph{measure constraint}.

While \eqref{FK} and \eqref{HKS} are well-known results, \eqref{FH} is more recent and some generalizations have been proven since then in \cite{bhmp, bdo, chp, gl}. All these works provide alternative proofs of \eqref{FH}, which, apart for \cite{chp}, are based on properties of Bessel functions. By allowing more nonlinear degrees of freedom, it has been found in \cite{n} that, for certain configurations, the optimal shape is not symmetric, given by a pair of non-equal balls. It is worth noting that, for a \emph{fixed} $g\in L^2_{\text{loc}}(\mathbb R^d)$, the minimization problem
\begin{equation}\label{p.finale}
\min_{\Omega\in \mathcal O} |\Omega|^{2/d} \lambda_1^g(\Omega)
\end{equation}
is not only a different problem, but it is not fully clear for which functions $g$ an optimal shape exists (this will be indeed the subject of a forthcoming paper; see also Corollary~\ref{c.nonexistence} at the end of the introduction).

Now, the equality \eqref{infg} also implies that a uniform lower bound (both with respect to the function $g$ and to the set $\Omega$) can be directly obtained \emph{via} \eqref{FK}. The situation becomes more interesting when one focuses on the double minimization problem 
\begin{equation}\label{problem}
\min_{\Omega\in \mathcal O}\min_{g \in L^\infty_\alpha} |\Omega|^\frac{2}{d}\lambda_{1}^{g}(\oo),
\end{equation}
with $0<\alpha\le 1$.
The first main result of the paper is an isoperimetric inequality for twisted eigenvalues: for every $0<\alpha\le 1$ the unique minimizer of \eqref{problem} is given by the union of two disjoint balls with orthogonality constraint of bang-bang type (see Figure~\ref{fig.opt} for some plots of the optimal shapes).

 \begin{thm}[Isoperimetric inequality for twisted eigenvalues]\label{t.main}
Fix $0< \alpha \le 1$. There exists a \emph{unique} number $m=m(\alpha)$, with $0< m(\alpha)\le 1$, such that for every set $\oo\in \mathcal O$ and every function $g \in L_\alpha^\infty$
        \begin{equation}\label{isoperimetric}
           |\Omega|^{\frac{2}{d}}\lambda_{1}^{g}(\oo)
            \ge |B_+^\alpha  \cup B_-^{\alpha}|^{\frac{2}{d}}\lambda_{1}^{\chi_\alpha}(B_+^{\alpha} \cup B_-^\alpha )=:\lambda(\alpha),
        \end{equation}
        where 
        $B_+^\alpha$ and $B_-^\alpha $ are any disjoint balls, $|B_-^\alpha |/|B_+^\alpha|=m(\alpha)$, 
and $\chi_\alpha=\alpha\chi_{B^{\alpha}_+} + \chi_{\mathbb R^d\setminus B^{\alpha}_+}$.
Equality holds in \eqref{isoperimetric}  if and only if $\oo=B_+^\alpha \cup B_-^\alpha $ and $g=\chi_\alpha$ a.e. in $B_+^\alpha\cup B_-^\alpha $.

Moreover, for the measure ratio $m(\alpha)$ of the optimal balls and for the minimal value $\lambda(\alpha)$ in \eqref{isoperimetric} the following estimates hold:
    \begin{equation}\label{dis.radius}
        m(\alpha) \ge \alpha^{\frac{d}{d-1}} \quad \text{and} \quad \lambda(\alpha)\ge \big(1+\alpha^\frac{d}{d-1}\big)^\frac{2}{d}|B|^\frac{2}{d}\lambda_1(B),
            \end{equation}
                    where $B$ is any ball in $\mathbb R^d$.

Finally, if $u_\alpha$ is a twisted eigenfunction corresponding to $\lambda_{1}^{\chi_\alpha}(B_+^\alpha  \cup B_-^\alpha )$, then the modulus of its gradient $|\nabla u_\alpha|$ is constant on the boundary $\partial B_+^\alpha  \cup \partial B_-^\alpha$  with
    \begin{equation}\label{eq.optimality}
|\nabla u_\alpha|^2={\frac{2\lambda_1^{\chi_\alpha}(B_+^\alpha  \cup B_-^\alpha )}{d|B_+^\alpha  \cup B_-^\alpha |}} \int_{B_+^\alpha  \cup B_-^\alpha } u_\alpha(x)^2 dx,\quad \text{ on $\partial B_+^\alpha  \cup \partial B_-^\alpha $} \, .
    \end{equation}
 \end{thm}
 
\begin{figure}[t]
\includegraphics[width=4.1cm]{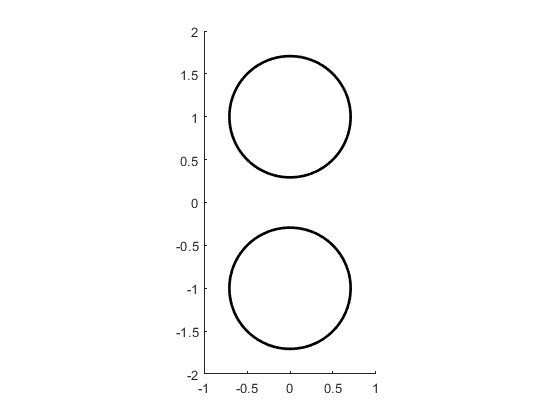}
\includegraphics[width=4.1cm]{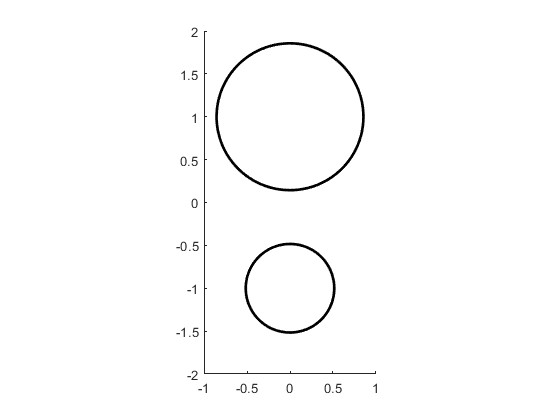}
\includegraphics[width=4.1cm]{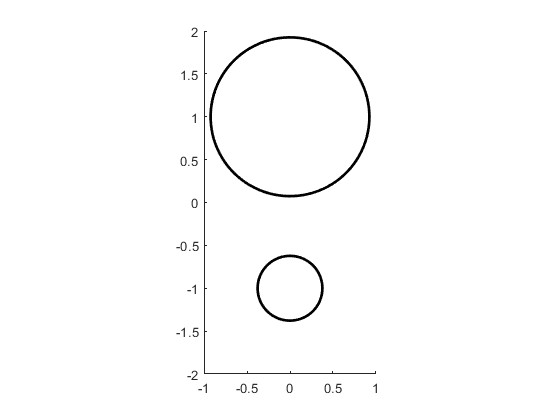}
        \caption{The optimal shape (up to scaling) $B_+^\alpha\cup B_-^\alpha$ corresponding to the values $\alpha=1$, $\alpha=\frac16$ and $\alpha=\frac{1}{20}$, in the case $d=2$. When the balls are different, the optimal bang-bang function $\chi_\alpha$ equals $\alpha$ in the larger ball and $1$ in the smaller one.}\label{fig.opt}
\end{figure}
 
The estimates \eqref{dis.radius} are derived via \emph{shape derivatives} of twisted eigenvalues without employing results on Bessel functions. Then, if $\alpha=1$, by \eqref{dis.radius} we have $m(1)=1$ and the Freitas-Henrot inequality \eqref{FH} follows, purely from the shape derivative of twisted eigenvalues, without relying on Bessel functions (the strategy is similar to \cite{bhmp} but at the end we do not use Bessel functions). Potentially, this may have applications to nonlinear eigenvalues, where an explicit representation of the eigenfunctions in terms of Bessel functions is not available (this was already noticed in \cite{chp} where a nonlinear variant of the problem was considered, however their proof required a function space larger than $H^1_0(\Omega)$, similar to the one introduced in \cite{gl}). The inequality with the lower bound for $\lambda_1^g(\Omega)$ in \eqref{interwinedTD} gives also the Hong-Krahn-Szego inequality.

\begin{cor}[Hong-Krahn-Szego and Freitas
Henrot inequalities]\label{cor.FH}
If $\alpha=1$ then \eqref{FH}, and in particular \eqref{HKS}, hold, together with their rigidity statements.
\end{cor}

Section~\ref{sec.4} contains the proofs of Theorem~\ref{t.main} and Corollary~\ref{cor.FH}. Notice, however, that extending the result to the general case  $\alpha<1$ requires proof strategies that rely on Bessel functions. 

As a second topic we study the dependence of the measure ratio $m(\alpha)$ and of the minimal value $\lambda(\alpha)$ in terms of the parameter $\alpha$ (see Figure~\ref{fig12} for some plots of these two maps).

 \begin{thm}[Continuous maps connecting the optimal shapes of Dirichlet eigenvalues]\label{t.main2}
 The functions $m(\alpha)$ and $\lambda(\alpha)$ in Theorem~\ref{t.main} are
 continuously differentiable in $(0,1)$, strictly increasing with  
        \begin{equation*}
             \lim_{\alpha \rightarrow 0^+} m(\alpha)=0  \quad \text{and} \quad  \lim_{\alpha \rightarrow 1^-} m(\alpha)=m(1)=1,
        \end{equation*}
        and moreover
          \begin{equation*}
       \lim_{\alpha \rightarrow 0^+} \lambda(\alpha)
        =|B^\wedge|^\frac{2}{d}\lambda_1(B^\wedge)
        \ \text{and} \
         \lim_{\alpha \rightarrow 1^-}\lambda(\alpha)= \lambda(1)
        = |B^\wedge \cup B^\vee|^\frac{2}{d}\lambda_{2}(B^{\wedge} \cup B^{\vee}), 
    \end{equation*}
    where $B^\wedge$ and $B^\vee$ are any disjoint balls in $\mathbb R^d$ of equal measure.
    \end{thm}

Section~\ref{sec.5} contains the proof of Theorem~\ref{t.main2}. Roughly speaking, the double minimization problem \eqref{problem} offers a way to interpolate between the optimal shapes of the first two Dirichlet eigenvalues of the Laplacian, as given in \eqref{FK} and \eqref{HKS}, through a \emph{continuous} 1-parameter family of optimal shapes that are \emph{union of two disjoint balls}. We are not aware of any results of this type that simultaneously maintain the continuity of the map and ensure that the optimal shapes are unions of balls. We would like to remind the reader that the optimal sets for non-trivial linear combinations of the first two Dirichlet eigenvalues are not represented by unions of balls, see \cite[Section 6.4.2, Open problem 20]{h} and also \cite[Theorems 1.1 and 1.3]{mz}. On the contrary, 
a connection between the optimal shapes for $\lambda_1$ and $\lambda_{1}^{T}$ was found in \cite{bfnt}, where the minimizing sets associated with a $1$-parameter family of eigenvalues are studied. This problem exhibits a saturation behavior: the minimal eigenvalue increases with the parameter up to a finite critical value, after which it remains constant. This critical point marks the transition between one ball to two equal balls as the optimal shapes. Therefore, the map that associates the parameter to the optimal shape is not continuous in the topology induced by the Hausdorff distance. Similar questions are addressed in \cite{dp}. We notice that another $1$-parameter family of balls gives the solution of an optimization problem involving the \emph{twisted Cheeger constant}, see \cite{bdpnt}.

\begin{figure}
\includegraphics[width=6.2cm]{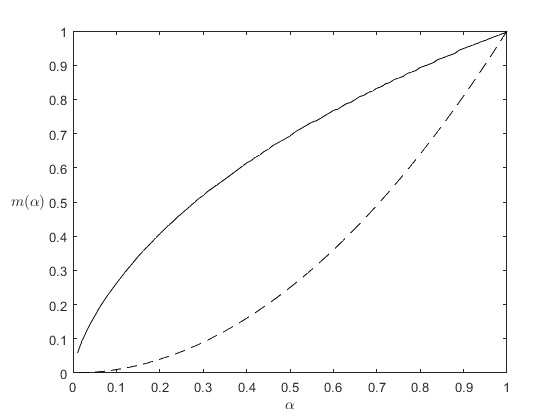}
\includegraphics[width=6.2cm]{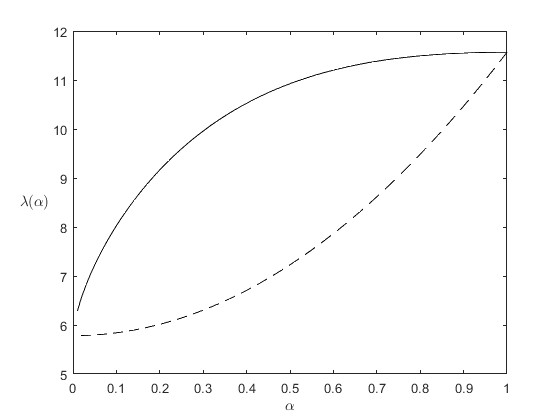}
        \caption{Plots of the functions $m(\alpha)$ and $\lambda(\alpha)$ (continuous lines) with their lower bounds in \eqref{dis.radius} (dashed lines), in the case $d=2$.}\label{fig12}
\end{figure}    

Now, we write a comment on the role of the dimension. On the one hand, the main results presented in this paper -- such as the isoperimetric inequalities -- also hold in the one-dimensional case ($d=1$), where the proofs simplify and may even yield stronger results. On the other hand, certain arguments encounter technical difficulties in this setting. For instance, the estimate  \eqref{dis.radius} becomes ill-defined in dimension one, and it is unclear how to recover an appropriate analogue.  For this reason, we have chosen to address the one-dimensional case separately in a future work (and, in the present paper, we mainly omitted the references related to the one-dimensional case).

We conclude with an interesting application of Theorems~\ref{t.main} and \ref{t.main2} to a non-existence result for the shape optimization problem \eqref{p.finale}, for a suitable function $g$ represented in Figure~\ref{fig.monotonecirc}. A minimizing sequence consists of two balls: a small ball centered at the origin whose radius shrinks, and a larger ball whose radius grows while its center drifts away from the origin. Its proof is contained in Section~\ref{sec.5}.

\begin{cor}[Non-existence result]\label{c.nonexistence}
If $\alpha \in (0,1)$ then there exists $g\in L^\infty_\alpha(\mathbb R^d)$ for which problem \eqref{p.finale} has no solution.
\end{cor}

\medskip
\paragraph{\textbf{Notation}}  
We denote by $\N$ the set of natural numbers
starting at one. We denote by $\mathbb{R}^+ := (0,+\infty)$ the set of positive real numbers. We use the shorthand notation \emph{a.e.} to say \emph{almost everywhere} with respect to the Lebesgue measure and \emph{w.r.t.} to say \emph{with respect to}. Given $a \in \R$ and $f \in L^2_{\text{loc}}(\R^d)$, for the sake of brevity, we simply write $f \ge a$ to denote $f\ge a$ a.e. in $\mathbb R^d$ (and similarly for the inequalities $>,<,\le$ or the equality $=$). Given a set $A\in\mathcal O$ the characteristic function $\chi_A$ of the set $A$ is defined as usual by $\chi_A(x)=1$ if $x\in A$ and $\chi_A(x)=0$ otherwise. The symbol $\delta_d$ denotes the volume of the ball of unit radius in $\mathbb R^d$. Given $p \in [1,+\infty]$ the symbol $||\cdot||_{p}$ is the $L^p$-norm of a function defined in $\R^d$. When there is no ambiguity, given a set of class $C^1$ we denote by $\nu$ its outward unit normal, defined on each point of its boundary. The symbol $\mathcal H^{d-1}$ denotes the $d-1$ dimensional Hausdorff measure.
  
If no assumptions are specified on $\Omega$, $g$, and $\alpha$ we always mean $\Omega \in \mathcal O$, $g\in L^2_\text{loc}(\mathbb R^d)$, and $0<\alpha\le 1$.
For any continuous function $u\in H^1_0(\Omega)$ we associate the positivity and negativity sets $\oo^+=\{u>0\}$ and $\oo^-=\{u<0\}$, and define the positive and negative parts of $u$ as $u^+=u|_{\oo^+}$ and $u^-=-u|_{\oo^-}$, respectively (so that $u=u^+ -u^-$). This should not be confused with the notation $B_+$ and $B_-$ (here $+$ and $-$ are subscripts) which always refers to two disjoint \emph{non-empty} balls. 
As usual, functions in $H^1_0(\Omega)$ are extended by $0$ outside $\Omega$.
When there is no ambiguity, given two sets $A_1,A_2\in \mathcal O$ we associate the bang-bang function $\chi_\alpha$, defined as $\chi_\alpha:=\alpha\chi_{A_1}+\chi_{\mathbb R^d\setminus A_1}$. If there is no misunderstanding, $\chi_\alpha$ is the function equal to $\alpha$ in $A_1$ and $1$ in $A_2$ (and extended to 1 elsewhere). By \emph{the} first Dirichlet eigenfunction $u_1$ of a simple eigenvalue $\lambda_1(\Omega)$ (e.g., when $\Omega$ is connected) we always refer to the \emph{positive} and \emph{normalized in $L^2(\Omega)$}, namely  $\int_{\Omega} u_1(x)^2 dx=1$, eigenfunction corresponding to $\lambda_1(\oo)$.

\section{Preliminary results on twisted eigenfunctions}\label{s.preliminary}

In this section we prove some properties of twisted eigenfunctions corresponding to $\lambda_1^g(\Omega)$. We use the shorthand notation $\mathcal R(u)$ to denote
the \emph{Rayleigh quotient} of a function $u\in H^1(\Omega)$, i.e., we define
\begin{equation*}
         \mathcal{R}(u):=\frac{\int_\oo |\nabla u(x)|^2  dx}{\int_\oo  u(x)^2  dx}.
\end{equation*}

We start by deriving the Euler-Lagrange equation satisfied by twisted eigenfunctions. To this purpose we need to construct a function in $H^1_{0,g}(\Omega)$ as a suitable linear combination of functions in $H^1_0(\Omega)$. Since this construction will be useful in several proofs, we state it separately in the following.

\begin{lem}\label{l.convex}
Let $v_1,v_2\in H^1_0(\Omega)$ and consider their linear combination $v$ defined by 
\begin{equation}\label{e.convex}
    v:=\begin{cases}
    v_1-\gamma v_2 \quad &\text{if $v_2\notin H^1_{0,g}(\Omega)$,}\\
    v_2\qquad &\text{if $v_2\in H^1_{0,g}(\Omega),$}
\end{cases}
\qquad \text{with } \gamma:=\frac{\int_\oo v_1(x) g(x)  dx}{\int_\oo v_2(x) g(x)  dx}.
\end{equation}
Then $v \in \gi$. If moreover $v_1$ and $v_2$ are \emph{linearly independent}  then $v\not\equiv 0$.
\end{lem}
\begin{proof}
Linear combinations of $H^1_0(\Omega)$ functions are in $H^1_0(\Omega)$. By definition of $v$ and the linearity of the integral, $\int_\Omega v(x)g(x) dx=0$, thus $v\in \gi$. Moreover, if the functions $v_1,v_2$ are linearly independent, then $v_1\not\equiv 0$, $v_2\not\equiv 0$, and also $v\not\equiv 0$.
\end{proof}

\begin{prop}\label{p.min}
Let $u$ be a minimizer of \eqref{lambda}. Then 
\begin{equation}\label{eq.lambda}
    \int_\oo \nabla u \nabla \phi \, dx=  \lambda_{1}^{g}(\oo) \int_\oo u \phi \, dx, \quad \text{for every } \phi \in \gi
\end{equation}
and
    \begin{equation}\label{eq.EL}
        \int_\oo \nabla u \nabla \psi \, dx
    =  \lambda_{1}^{g}(\oo) \int_\oo u \psi \, dx+ \xi \int_\oo g \psi \, dx, \quad \text{for every } \psi \in H^1_{0}(\oo) ,
    \end{equation}
for some $\xi \in \R$ (possibly depending on $u$ and $g$, i.e., $\xi =\xi(u,g)$). If $g\not\equiv 0$ and $u$ also belongs to $H^2(\oo)$, then it satisfies \eqref{pde} in the weak sense, namely \eqref{eq.EL} with
    \begin{equation}\label{eq.xi}
            \xi=\dfrac{\int_{\oo} (-\Delta u(x))g(x)dx}{\int_{\Omega} g(x)^2 dx}.
    \end{equation}
\end{prop}
\begin{proof}
Let $\lambda=\lambda_{1}^{g}(\oo)$. The existence of a minimizer follows by the Direct Methods of the Calculus of Variations (by
taking a suitable minimizing sequence that, up to scaling, can be assumed to be normalized in $L^2(\Omega)$). The space $\gi$ is compact with respect to the weak convergence in $H^1(\Omega)$ (notice that the orthogonality condition to $g$ is preserved by this convergence) and the Rayleigh quotient $\mathcal{R}$ is lower semicontinuous with respect to this convergence.  Now, if $g\equiv 0$ there is nothing to prove, since $\lambda$ is the first Dirichlet eigenvalue. Let $g\not\equiv 0$ and $u\in \gi\setminus\{0\}$ be a minimizer of \eqref{lambda} then
\begin{equation*}
    \dfrac{d}{d \epsilon} \left(\mathcal R(u+\epsilon\phi)\right) \Big|_{\epsilon=0}=0, \quad \textup{for every } \phi \in \gi,
\end{equation*}
and \eqref{eq.lambda} follows.
For $\psi \in H^1_{0}(\oo)$ and $\vp \in H^1_{0}(\oo) \setminus \gi$, by Lemma~\ref{l.convex}, we can consider the function
\begin{equation*}
    \phi:=\psi- \frac{\int_\oo g(x) \psi(x)  dx}{\int_\oo g(x) \vp(x)  dx} \vp\in \gi,
\end{equation*}
which plugged into \eqref{eq.lambda} yields
\begin{equation*}
    \int_\oo \nabla u(x) \nabla \psi(x)  dx 
    = \lambda \int_\oo u(x) \psi(x)  dx +  \xi \int_\oo g(x) \psi(x)  dx,\  \text{ for every } \psi \in H^1_{0}(\oo),
\end{equation*}
with
\begin{equation*}
    \xi=\xi(u,g,\vp)=\frac{\int_\oo \nabla u(x)  \nabla \vp(x)  dx - \lambda \int_\oo u(x) \vp(x)  dx}{\int_\oo g(x) \vp(x)  dx}.
\end{equation*}
This is equation \eqref{eq.EL} and it remains to prove that the quantity $\xi$ does not depend on $\varphi$.
To this purpose, let $\vp_1,\vp_2, \psi_0 \in H^1_{0}(\oo) \setminus \gi$. Subtracting the equation above with $\vp=\vp_1$ and $\psi=\psi_0$ from the one with $\vp=\vp_2$ and $\psi=\psi_0$ we obtain
\begin{equation*}
    [\xi(u,g,\vp_1)-\xi(u,g,\vp_2)]\int_\oo g(x) \psi_0(x)  dx=0.
\end{equation*}
We chose $\psi_0$ such that $\int_\oo g(x) \psi_0(x)  dx\neq 0$, hence $\xi(u,g,\vp_1)=\xi(u,g,\vp_2)=\xi(u,g)$.

Finally, when $g\not\equiv 0$ and $u$ also belongs to $H^2(\oo)$, an integration by parts yields
\begin{equation}\label{eq.dim2}
    \int_{\oo} \nabla u \nabla \psi  dx
    = \lambda \int_{\oo} u \psi  dx+ \frac{\int_{\oo} (-\Delta u)  \vp  dx - \lambda \int_{\oo} u \vp  dx }{\int_{\oo} g \vp  dx} \int_{\oo} g \psi  dx.
\end{equation}
To finish the proof we would take $\vp=g$ in \eqref{eq.dim2}, which is clearly not possible, since $g$ could not belong to $H^1_{0}(\Omega)$. However, this can be justified with a limiting argument: by density, there exists a sequence $(\vp_n)_n$ in $H^1_{0}(\Omega)$ such that $\vp_n \rightarrow g$ in $L^2(\oo)$ as $n \rightarrow +\infty$. Moreover, for every $n$ sufficiently large such that $\|\vp_n-g\|_2<\|\vp_n\|_2$ (this is always possible since the left-hand side converges to 0 while the right-hand side goes to $\|g\|_2>0$ being $g\not\equiv0$) it follows that
$$2\int_\Omega g(x) \vp_n(x) dx=\|\vp_n\|_2^2+\|g\|_2^2 -\|\vp_n-g\|_2^2>\|g\|^2_2>0,$$
that is there exists a sequence $(\vp_n)_n$ in $H^1_{0}(\oo) \setminus \gi$ converging to $g$ in $L^2(\Omega)$. By letting first $\vp=\vp_n$ in \eqref{eq.dim2} and then $n \rightarrow +\infty$ we find that a minimizer solves \eqref{pde} in the weak sense.
\end{proof}

We aim to characterize the sign of $\xi$ in \eqref{eq.EL}. A preliminary result is the following.

\begin{oss}\label{bbs33}
Let $u$ be a twisted eigenfunction corresponding to $\lambda_1^g(\Omega)$. If $u$ is not a Dirichlet eigenfunction and $g$ does not vanish a.e. on any open subset of $\Omega$, then $u$ cannot vanish on any open subset of $\Omega$. Indeed, assume that $u \equiv 0$ in some open set $A\subset \Omega$. Then, since $g\not\equiv 0$ a.e. in $A$ by the du Bois-Reymond Lemma there exists $\psi_0\in H^1_0(A)$ such that $\int_A \psi_0 g dx \neq 0$ and by plugging $\psi_0$ in \eqref{eq.EL} one obtains $\xi \equiv 0$. Therefore, $u$ satisfies \eqref{eq.EL} with $\xi=0$, that is $u$ is a Dirichlet eigenfunction of $\Omega$ (which vanishes in $A$). This proves what claimed above. 

A particular case of application is when $\Omega$ is not connected with $A$ one of its connected components and $g$ piecewise constant on the connected components. 
\end{oss}

To obtain more information, we need to state some regularity results for twisted eigenfunctions corresponding to $\lambda_{1}^{g}(\oo)$, depending on the regularity of $\oo$ and of $g$.

\begin{lem}\label{l.reg}
Let $u$ be a twisted eigenfunction corresponding to $\lambda_1^g(\Omega)$. 
\begin{itemize}
\item[(a)] If $g\in  L^\infty(\mathbb R^d)$, then $u$ is continuous in $\Omega$.
\item[(b)] If $g$ is analytic in $\oo$, then $u$ is analytic in $\oo$.
\item[(c)] If $g \in C^\infty(\overline{\oo})$ and $\oo$ is a set of class $C^\infty$, then $u \in C^\infty(\overline{\oo})$. In particular, $u$ satisfies \eqref{pde} in the classical sense.
\item[(d)] If $\Omega$ is convex or is of class $C^{1,1}$, then $u \in H^2(\Omega)$.
\end{itemize}
\end{lem}

\begin{proof}
All items follow from classical elliptic regularity results, see for instance \cite{bjs,eva,gt}. More precisely, if $u$ satisfies \eqref{eq.EL} with $g$ bounded,  by \cite[Theorem 8.22]{gt}, there holds (a).
Moreover, by \cite[Analyticity Theorem, p.~136 and Theorem 8, p.~200]{bjs} and \cite[Theorem 2, p. 273]{eva} it follow the remaining items (b) and (c). Finally, by \cite[Theorem 3.2.1.2 and Theorem 2.4.2.5]{gris} there holds item (d).
\end{proof}

\begin{oss}\label{r.atleast}
By item (a) of Lemma~\ref{l.reg} we can always consider the positive and negative parts $u^+$ and $u^-$ of $u$ with their positivity and negativity sets $\Omega^+$ and $\Omega^-$. Of particular importance is the role played by the space of \emph{positive} uniformly bounded functions $L^\infty_+$.
If $g\in L^\infty_+$ then every $u \in \gi$ must change sign in $\oo$, that is neither $u^+$ nor $u^-$ is identically zero, since $\int_{\oo}  gu  \, dx=0$. In particular, every twisted eigenfunction has at least two nodal domains. Recall that the \textit{nodal domains} of a continuous function $u$ defined over $\Omega$ are the connected components of the open set $\Omega\setminus u^{-1}(0)$.
\end{oss}

We first characterize the sign of $\xi$ in \eqref{eq.EL} in terms of the magnitude of the Rayleigh quotients of twisted eigenfunctions restricted to the positivity and negativity sets.

\begin{prop}\label{prop.xi}
If $g\in L^\infty_+$ and $u$ is a twisted eigenfunction corresponding to $\lambda_{1}^{g}(\oo)$ then 
    \begin{equation*}
        \mathcal{R}(u^+) \le \mathcal{R}(u^-)
        \quad \text{if and only if} \quad
       \xi\le 0.
    \end{equation*}
The equality $\mathcal{R}(u^+)=\mathcal{R}(u^-)$ holds if and only if $\xi=0$ (i.e., $u$ is a Dirichlet eigenfunction).
\end{prop}
\begin{proof}
By letting $\psi=u^+/\|u^+\|_2^2$ in \eqref{eq.EL}, the function $u^+$ satisfies
       \begin{equation*}
          \mathcal{R}(u^+)
    =  \lambda_{1}^{g}(\oo) + \frac{\xi}{||u^+||^2_{2}} \int_\oo g(x) u^+(x)  dx,
     \end{equation*}
        while, by letting $\psi=u^-/\|u^-\|_2^2$ in \eqref{eq.EL}, the function $u^-$ satisfies
          \begin{equation*}
          \mathcal{R}(u^-)
    =  \lambda_{1}^{g}(\oo) - \frac{\xi}{||u^-||^2_{2}} \int_\oo g(x) u^-(x)  dx.
     \end{equation*}
Subtracting the second equation from the first, and using the identity $\int_\oo g u^+dx =\int_\oo gu^-dx$, which follows from the condition $u \in \gi$, yields
\begin{equation*}
    \mathcal{R}(u^+) - \mathcal{R}(u^-)
    =\xi \int_\oo g(x) u^+(x)  dx
    \left( \frac{1}{||u^+||^2_{2}} + \frac{1}{||u^-||^2_{2}} \right) ,
\end{equation*}
and the thesis follows from the positivity of $g$ and $u^+$.
\end{proof}

We will also need the following elementary result.

\begin{lem}\label{l.algebraic}
    Let $a_1,a_2,b_1,b_2>0$ with $a_1/a_2 \ge b_1/b_2$. The function $Q\colon \mathbb R^+\to \mathbb R^+$ defined by
    \begin{equation*}
       Q(t)=\frac{a_1t+b_1}{a_2t+b_2}
    \end{equation*}
    is non-decreasing in $\mathbb R^+$. The function $Q$ is strictly increasing if and only if 
    $a_1/a_2 > b_1/b_2$ (otherwise it is constant).
     In particular, there hold
    \begin{equation*}
       \frac{b_1}{b_2}\le \frac{a_1+b_1}{a_2+b_2} \le \frac{a_1}{a_2},
    \end{equation*}
    with equalities if and only if $a_1/a_2 = b_1/b_2$. 
\end{lem}

\begin{oss}\label{td}
Let $\Omega=B_+\cup B_-$ where $B_+,B_-$ are disjoint balls with $|B_+|\geq  |B_-|$ and $g=\chi_\alpha=\alpha \chi_{B_+}+\chi_{\mathbb R^d\setminus B_+}$. By \eqref{interwinedTD} and \cite[Remark 1.2.4]{h} there hold
\[
\lambda_1(B_+)\leq \lambda_1^{\chi_\alpha}(B_+\cup B_-)\leq \min\{\lambda_1(B_-),\lambda_2(B_+)\}.
\]
Only in the special case where equality holds in the previous inequalities can a twisted eigenfunction $u$ be a Dirichlet eigenfunction. Indeed, by Proposition~\ref{prop.xi} and Lemma~\ref{l.algebraic}, we have the following necessary and sufficient conditions:
\begin{itemize}
\item[(i)] $\lambda_1^{\chi_\alpha}(B_+\cup B_-)=\lambda_1(B_+)$ or $\lambda_1^{\chi_\alpha}(B_+\cup B_-)=\lambda_1(B_-)$ if and only if $B_+$ and $B_-$ are such that $|B_+|=|B_-|$ with $\lambda_1^{\chi_\alpha}(B_+\cup B_-)=\lambda_2(B_+\cup B_-)$ if and only if $u$ is equal, up to scaling, to $u_1^+-\alpha u_1^-$ (with $u_1^\pm$ the first Dirichlet eigenfunctions in $B_\pm$);
\item[(ii)] $\lambda_1^{\chi_\alpha}(B_+\cup B_-)= \lambda_2(B_+)$ if and only if $u$ is equal, up to scaling, to a Dirichlet eigenfunction $u_2^+$ corresponding to $\lambda_2(B_+)$.
\end{itemize}
\end{oss}

Now, we need a result \emph{\`a la} Courant, concerning the number of nodal domains of twisted eigenfunctions (see Remark~\ref{r.atleast}). This is a very delicate issue and we are able to determine the exact number of nodal domains of twisted eigenfunctions when $\oo$ is the union of two balls and $g$ is of bang-bang type.

\begin{thm}[Nodal domain theorem]\label{simpleCourant}
Let $\Omega=B_+\cup B_-$ where $B_+,B_-$ are disjoint balls with $|B_+|\ge |B_-|$ and $g=\chi_\alpha=\alpha \chi_{B_+}+\chi_{\mathbb R^d\setminus B_+}$. If 
\begin{equation}\label{eq.OmBT<D2}
    \lambda_1^{\chi_\alpha}(B_+\cup B_-) < \lambda_2(B_+),
\end{equation}
then the eigenvalue $\lambda_{1}^{\chi_\alpha}(B_+\cup B_-)$ is simple, the corresponding eigenfunction $u$ is (up to non-zero scalar multiples) radially symmetric, positive, and monotone decreasing (w.r.t. the distance from the center of $B_+$) in
$B_+$, while radially symmetric, negative, and monotone increasing (w.r.t. the distance from the center of $B_-$) in $B_-$, with exactly two nodal domains $\Omega^+=B_+$ and $\Omega^-=B_-$.
\end{thm}
\begin{proof}

Let $u$ be an eigenfunction corresponding to $\lambda:=\lambda_1^{\chi_\alpha}(B_+\cup B_-)$ and denote $u_\pm:=u|_{B_\pm}$ the restrictions of $u$ to each ball. The case $|B_+|=|B_-|$ has already been contained in item (i) of Remark~\ref{td}. Therefore, we may only consider the case $|B_+|>|B_-|$ and divide the proof in three steps.
\smallskip

\emph{Step 1 ($u$ has a nodal domains in $B_+$ and another in $B_-$).}
By Remark~\ref{td} we know that $u$ is not a Dirichlet eigenfunction (i.e. $\xi\neq 0$ in \eqref{eq.EL}) with
\begin{equation}\label{fine}
   \lambda_1(B_+)< \lambda<\lambda_1(B_-).
\end{equation} 
Then, by Remark~\ref{bbs33}, $u$ does not vanish identically on any of the connected components of $B_+\cup B_-$.
In particular $u$ has at least a nodal domains in $B_+$ and another in $B_-$.  

\smallskip
\emph{Step 2 ($u$ is radially symmetric w.r.t. each center of the balls $B_+$ and $B_-$).} By item (b) of Lemma~\ref{l.reg} $u$ is $C^\infty$ in $B_+\cup B_-$ and it satisfies
 \begin{equation}\label{eq.uCourant}
     -\Delta u = \lambda u +\xi\chi_\alpha \quad \text{in } B_+\cup B_-.
 \end{equation} 
Let $I_+$ be an isometry of $\R^d$ with the same center of $B_+$. Let $v_{+}$ be the function obtained by moving isometrically the function $u$ with $I_+$ on the ball $B_+$, that is
\[
v_+(x)= u_+(I_+(x)) + u_-(x),\quad \text{for every $x\in B_+\cup B_-$.} 
\]
 Since $v_+ \in \gi$ and $\mathcal{R}(v_+)=\mathcal{R}(u)$ then $v_+$ is an eigenfunction corresponding to $\lambda$ that satisfies the equation \eqref{eq.uCourant} with the same $\xi$. If, by contradiction, $u-v_{+} \not \equiv 0$ in $B_+\cup B_-$, then by \eqref{eq.uCourant} $u-v_+$ is a Dirichlet eigenfunction of $B_+\cup B_-$ with corresponding eigenvalue $\lambda$. In particular $u_+-(u_+ \circ I_+)$ is a Dirichlet eigenfunction of $B_+$ with corresponding eigenvalue $\lambda$, a contradiction with \eqref{eq.OmBT<D2} and \eqref{fine}. This implies that $u_+$ is radially symmetric in $B_+$. A similar argument applies for $u_-$.

 \smallskip
 \emph{Step 3 ($u$ has nodal domains $B_+$ and $B_-$).}
 Suppose by contradiction that there exists $r>0$ such that on the sphere $\partial B_r \subset B_+$ (with the ball $B_r$ of radius $r$ with same center of $B_+$ strictly contained in $B_+$) the radial derivative of $u_+$ is equal to $0$ (note that this is true in the case of $u_+$ with multiples nodal domains). Now, consider the function $w_+$ defined as
 \[
 w_+(x)=u_+(x)-u_+(x_r) \quad \text{for every $x\in \overline{B_r}$},
 \]
 where $x_r$ is a point on $\partial B_r$. 
By definition, $w_+$ is analytic in $B_r$ and then, by \eqref{eq.uCourant} and the construction of $r$, it satisfies
\[
\begin{cases}
\Delta^2 w_+=-\lambda \Delta w_+, \quad &\text{in } B_r,\\
w_+=(\partial w_+)/(\partial \nu)=0, \quad &\text{on $\partial B_r$}.
\end{cases}
\]
Therefore, $\lambda$ is also a \textit{buckling eigenvalue} of $B_r$ (not only a twisted one), see \cite[Section 15.1, equations (15.1.1)-(15.1.2)]{bw}. We can use the \textit{Payne inequality}, see \cite[Section 15]{bw} or \cite[Section 2, Part A]{pay}, to obtain 
$$\lambda \ge \Lambda_1(B_r) \ge\lambda_2(B_r)>\lambda_2(B_+),$$
where $\Lambda_1(B_r)$ is the first buckling eigenvalue of $B_r$. 
This last inequality contradicts the assumption \eqref{eq.OmBT<D2} and hence the radial derivative of $u_+$ has the same sign in every point of $B_+$. An analogous argument applies also to $u_-$. All in all, it follows that $u_+$ and $u_-$ have constant sign and are radially decreasing with respect to the centers of $B_+$ and $B_-$ respectively.

By Remark~\ref{r.atleast} the function $u$ has to change sign in $B_+\cup B_-$, hence we may assume $u>0$ in $B_+$ and $u<0$ in $B_-$. In particular, this implies that $\lambda$ is a simple eigenvalue. 
\end{proof}

We conclude this section with an interesting result that allows us to determine the sign of $\xi$ in \eqref{eq.EL} when the set $\Omega$ is the union of two disjoint balls and the function $g$ is of bang-bang type. 

\begin{prop}\label{p.rayleigh}
Let $\Omega=B_+\cup B_-$ where $B_+,B_-$ are disjoint balls with $|B_+|\ge |B_-|$ and $g=\chi_\alpha=\alpha \chi_{B_+}+\chi_{\mathbb R^d\setminus B_+}$. If $u$ denote a twisted eigenfunction corresponding to $\lambda_{1}^{g}(\oo)$
with two nodal domains $\Omega^+, \Omega^-$, then 
    \begin{equation*}
|\Omega^+|\ge |\Omega^-| \quad \text{if and only if}\quad  \mathcal{R}(u^{+})\le  \mathcal{R}(u^{-}).
    \end{equation*}
The equality $|\Omega^+|=|\Omega^-|$ holds if and only if $\mathcal{R}(u^+)=\mathcal{R}(u^-)$ (i.e., $u$ is a Dirichlet eigenfunction).
\end{prop}
\begin{proof}
We first prove the cases of equality. By Proposition~\ref{prop.xi} we know that the equality $\mathcal{R}(u^{+})=  \mathcal{R}(u^{-})$ holds if and only if the twisted eigenfunction is a Dirichlet eigenfunction which, by Remark~\ref{td}, is true if and only if either $\lambda_{1}^{\chi_\alpha}(B_+\cup B_-)=\lambda_2(B_+)$ or  $\lambda_{1}^{\chi_\alpha}(B_+\cup B_-)=\lambda_1(B_+)$. In all cases this turns out equivalent to $|\oo^+|=|\oo^-|$.

Now, in order to prove the double implication for the strict inequalities in the statement, we may assume $|\Omega^+|\neq |\Omega^-|$ which implies, in particular, $\lambda_{1}^{\chi_\alpha}(B_+\cup B_-)<\lambda_2(B_+)$. By Theorem~\ref{simpleCourant}, up to exchanging the sets $\Omega^+$ and $\Omega^-$, it follows that $\Omega^+=B_+$ and $\Omega^-=B_-$; that is $u$ is non-zero in both balls. The strategy of the proof is to test \eqref{lambda} with suitable combinations of $u^+$ or $u^-$. We proceed in two steps.

\smallskip
\emph{Step 1 (upper bounds for the twisted eigenvalue).}
We first use a suitable combination of $u^-$ in the two balls, where on $B_+$ it is appropriately (translated and) scaled so as to belong to $H^1_{0,\chi_\alpha}(B_+\cup B_-)$, and hence is admissible for \eqref{lambda}.
Let $c_+$ and $c_-$ be the centers of the balls $B_+$ and $B_-$. For every $x\in B_+$ define the function $v^{+}(x):={u}^{-}((x-c_++c_-)/a)$ with $a:=(|B_+|/|B_-|)^{1/d}$ (i.e., $v^+$ is $u^-$ rescaled by the factor $a$ and translated so as to belong to $H^1_0(B_+$)). By using Lemma~\ref{l.convex} with $v_1=v^+$ and $v_2=u^-$, then $\gamma=\alpha a^d$ and the linear combination $v:=v^+-\alpha a^d u^-\in H^1_{0,\chi_\alpha}(B_+\cup B_-)$ 
can be plugged into \eqref{lambda} to yield
\begin{equation*}
\begin{split}
\lambda_1^{\chi_\alpha}(B_+\cup B_-)&\le \frac{\int_{B_+}|\nabla v^+|^2dx+ \alpha^2a^{2d} \int_{B_-}|\nabla u^-|^2dx}{ \int_{B_+}(v^+)^2dx+\alpha^2a^{2d} \int_{B_-}(u^-)^2dx}\\&=\frac{a^{d-2}\int_{B_-}|\nabla u^-|^2dx+ \alpha^2 a^{2d}\int_{B_-}|\nabla u^-|^2dx}{a^{d}\int_{B_-}(u^-)^2dx+\alpha^2 a^{2d}\int_{B_-}(u^-)^2dx}=\frac{a^{-2}+\alpha^2a^d}{1+\alpha^2a^d}\mathcal R(u^-),
\end{split}
\end{equation*}
where the first equality is obtained by changing variable with scaling factor $a$.
 Similarly, we may use $v:=u^+-\alpha a^d v^-$ with $v^{-}(x):={u}^{+}(a(x-c_-+c_+))$ for every $x\in B_-$ to obtain
\begin{equation}\label{dim22}
\lambda_1^{\chi_\alpha}(B_+\cup B_-)\le \frac{1+\alpha^2a^{d+2}}{1+\alpha^2a^d}\mathcal R(u^+).
\end{equation}

\smallskip
\emph{Step 2 (double implication for the strict
inequalities).} If $|B_+|>|B_-|$, namely $a>1$, then from the previous inequality $\lambda_1^{\chi_\alpha}(B_+\cup B_-)< \mathcal R(u^-)$. The estimates in Lemma~\ref{l.algebraic} yield
\begin{equation}\label{minmax}
\min\{\mathcal R(u^+),\mathcal R(u^-)\}\le\lambda_1^{\chi_\alpha}(B_+\cup B_-)\le\max\{\mathcal R(u^+),\mathcal R(u^-)\},
\end{equation}
where, in Lemma~\ref{l.algebraic}, $\mathcal R(u^+)=b_1/b_2<a_1/a_2=\mathcal R(u^-)$. On the other hand, if $\mathcal R(u^+)< \mathcal R(u^-)$, then \eqref{minmax} is equivalent to $\mathcal R(u^+)<\lambda_1^{\chi_\alpha}(B_+\cup B_-)$ and combined with \eqref{dim22} yields $(1+\alpha^2a^d)< (1+\alpha^2a^{d+2})$, that is $a>1$.
\end{proof}

By combining Propositions~\ref{prop.xi} with \ref{p.rayleigh} one may deduce the sign of $\xi$ in \eqref{eq.EL} for an eigenfunction with exactly two nodal domains, that is when $\Omega$ is the union of two balls and $g$ is of bang-bang type.

\begin{cor}\label{cor.xi}
Let $\Omega=B_+\cup B_-$ where $B_+$ and $B_-$ are disjoint balls with $|B_+|\ge|B_-|$ and $g=\chi_\alpha=\alpha \chi_{B_+}+\chi_{\mathbb R^d\setminus B_+}$. If $u$ is a twisted eigenfunction corresponding to $\lambda_{1}^{\chi_\alpha}(B_+\cup B_-)$
with two nodal domains $\Omega^+, \Omega^-$ such that $|\Omega^+|\ge |\Omega^-|$,  
 then $\xi\le 0$. 
\end{cor}

\section{Estimates of twisted eigenvalues with Dirichlet eigenvalues}\label{sec.g}

We prove bounds for twisted eigenvalues in terms of Dirichlet ones. We separately analyze the dependence on the function $g$ and on the shape $\Omega$.
 
\subsection{Estimates in terms of the function $g$}

We give necessary and sufficient conditions on the function $g$ under which the lower bound for $\lambda_1^g(\Omega)$ in \eqref{interwinedTD} is attained. We start with a sufficient condition. 

\begin{prop}\label{p.sufficient}
For all $a_1,a_2 \in (0,+\infty]$ there exists $g\in L^\infty(\mathbb R^d)$ with $|\{g>0\}|=a_1$ and $|\{g<0\}|=a_2$ such that 
    \begin{equation*}
         \lambda_{1}^{g}(\oo)=\lambda_1(\oo).
    \end{equation*}
\end{prop}
\begin{proof}
If $\Omega$ is connected then consider the first Dirichlet eigenfunction $u_1$ be of $\Omega$ and two disjoint open sets $A_1,A_2\subset \mathbb R^d$ such that $|A_1|=a_1$, $|A_1 \cap \oo|>0$, $|A_2|=a_2$, $|A_2 \cap \oo|>0$. If
    \begin{equation*}
        g=\Big(\int_{A_2 \cap \oo} u_1(x)  dx\Big) \chi_{A_1} -\Big(\int_{A_1 \cap \oo} u_1(x)  dx\Big)\chi_{A_2} \, ,
    \end{equation*}
    then $u_1 \in \gi$. Plugging $u_1$ into \eqref{lambda} and using \eqref{interwinedTD} gives the thesis. 
    
If $\Omega$ is not connected, by \cite[Remark 1.2.4]{h}, an analogous argument applies by restricting to a connected component whose first Dirichlet eigenvalue equals $\lambda_1(\Omega)$.
\end{proof}

We discuss a necessary condition in the case of connected sets.

\begin{prop}\label{p.necessary}
If $\oo$ is connected and $g\in L^\infty_+$, then $$ \lambda_{1}^{g}(\oo)>\lambda_1(\oo).$$
\end{prop}
\begin{proof}
We proceed by contradiction, assuming that $\lambda_{1}^{g}(\oo)=\lambda_1(\oo)$. Let $u_1$ be the first Dirichlet eigenfunction of $\oo$, and let $u$ be a normalized in $L^2(\Omega)$  twisted eigenfunction corresponding to $\lambda_{1}^{g}(\oo)$, chosen such that $\int_\Omega u(x)u_1(x) dx \ge 0$. Since the first Dirichlet eigenvalue is simple on the connected set $\Omega$, and by the assumption $\lambda_{1}^{g}(\oo)=\lambda_1(\oo)$, it follows that $u=u_1$ (recall the choice $\int_\Omega u u_1dx\ge0$). Moreover, it is known that $u_1$ does not change sign (see \cite[Theorem 6.34]{bor}). Therefore,
    \begin{equation*}
        \int_{\oo} u(x) g(x)  dx= \int_{\oo} u_1(x) g(x)  dx> 0,
    \end{equation*}
which contradicts the fact that $\int_\Omega u(x)g(x)dx=0$. The thesis follows from \eqref{interwinedTD}.
\end{proof}

\begin{oss}\label{r.necessary}
Proposition~\ref{p.necessary} is no longer true if the connectedness assumption on $\Omega$ is removed: there exists a disconnected set $\Omega\in\mathcal O$ and $g\in L^\infty_+$, for which the lower bound for $\lambda_1^g(\Omega)$ in \eqref{interwinedTD} is attained. For instance, one may consider $\oo=B^\wedge \cup B^\vee$, where $B^\wedge, B^\vee$ are two disjoint balls of equal measure and then \eqref{eq.ind} holds.
\end{oss}

Proposition~\ref{p.sufficient} shows that the lower bound for $\lambda_1^g(\Omega)$ in \eqref{interwinedTD} can be attained by a bounded function $g$ that is allowed to change sign.
In view of Proposition~\ref{p.necessary}, one may wonder whether the infimum of $\lambda_1^g(\Omega)$, when restricted to positive functions $g\in L^\infty_+$, is strictly greater that $\lambda_1(\Omega)$. The following proposition addresses this question and shows that the answer is negative.

\begin{prop}\label{p.espilon}
Given $\epsilon>0$, there exist $\alpha>0$ and $g\in L_\alpha^\infty$ such that
    \begin{equation*}
        \lambda_{1}^{g}(\oo) \le \lambda_1(\Omega) + \epsilon \, .
    \end{equation*} 
\end{prop}
\begin{proof}
 Assume that $\Omega$ is connected. For a given $x_0\in\partial \Omega$ consider the open ball $B(r)$ centered at $x_0$ of radius $r>0$ sufficiently small such that 
  \begin{equation}\label{hip}
         \lambda_{1}(\Omega\setminus \overline B(r)) \le \lambda_1(\Omega) +\epsilon/2 \,,
\end{equation}
(by the continuity of the first Dirichlet eigenvalue, see for instance \cite[Theorem~3.2.7 and Corollary 4.7.4]{hp}, $\lambda_1(\Omega \setminus \overline B(\rho)) \rightarrow \lambda_1(\Omega)$ as $\rho \rightarrow 0^+$ and such an $r$ always exists). Let $u_r^{+}$ and $u_r^{-}$ be, respectively, two first Dirichlet eigenfunctions corresponding to $\lambda_1(\Omega\setminus \overline{B}(r))$ and $\lambda_1(\Omega\cap B(r))$. For a given $\alpha>0$, let $g=\alpha\chi_{\Omega\setminus \overline{B}(r)} + \chi_{\mathbb R^d\setminus(\Omega\setminus \overline{B}(r))}$ and consider as in \eqref{e.convex} the linear combination $u_r:=u_{r}^+-\gamma u_{r}^-$ with $\gamma=\alpha\|u_r^+\|_1/\|u_r^-\|_1$. By Lemma~\ref{l.convex} $u_r\in\gi$, thus it can be used as test function in \eqref{lambda} yielding
    \begin{equation*}
        \lambda_{1}^{g}(\Omega)
        \le \mathcal R(u_r)= \frac{\int_\Omega |\nabla u_{r}^+(x)|^2dx  +\gamma^2 \int_\Omega |\nabla u_{r}^-(x)|^2dx }{\int_\Omega u_r^+(x)^2 dx+\gamma^2 \int_\Omega u_{r}^-(x)^2dx}.
    \end{equation*}
As $\alpha \rightarrow 0^+$, by definition $\gamma\to 0^+$, and the quantity in the right-hand side tends to $\lambda_1(\Omega\setminus \overline B(r))$. Therefore, there exists a sufficiently small $\alpha>0$ such that
    \begin{equation*}
        \lambda_{1}^{g}(\Omega) 
        \le \lambda_1(\Omega\setminus \overline B(r)) +\epsilon/2.
    \end{equation*}
     One obtains the thesis by combining the previous inequality with \eqref{hip}.
     
     As in the proof of the previous Proposition~\ref{p.necessary}, if $\Omega$ is not connected, by \cite[Remark 1.2.4]{h}, an analogous argument applies by restricting to a connected component whose first Dirichlet eigenvalue equals $\lambda_1(\Omega)$.
\end{proof}

As a consequence of the arbitrariness of $\epsilon$ in Proposition~\ref{p.espilon} it follows that
\begin{equation}\label{infL+}
    \inf_{g\in L^\infty_+} \lambda_1^g(\Omega) =\lambda_1(\Omega).
\end{equation}
Moreover, by Proposition~\ref{p.necessary}, if $\Omega$ is connected, the infimum is not a minimum -- that is, it is not attained by any bounded positive function. This equality also implies that a uniform lower bound for $\lambda_1^g(\Omega)$, including with respect to variations of $\Omega$, can be obtained via \eqref{FK}. The situation becomes more interesting when the infimum is restricted to the class $L^\infty_\alpha$ of \emph{uniformly positive} and \emph{uniformly bounded} functions, see \eqref{Lalpha}.

\begin{oss}\label{scalclass}
In order to carry out a finer analysis, to determine when the minimum is not attained in \eqref{problem}, it could be useful to consider the constrained formulation and consider the class 
$$\{ g \in L^\infty_{\alpha} \, : \, |\{g=\alpha\}| \le c \},$$
for some $c>0$.
Notice that this family is not closed under scalings, in contrast to $L^\infty_{\alpha}$.
\end{oss}

We conclude this subsection with a monotonicity result involving bang-bang functions as orthogonality constraints in the class $L^\infty_\alpha$.

\begin{prop}\label{p.monotonicity2}
Let $g\in L^\infty_\alpha$ and $u$ be a twisted eigenfunction corresponding to $\lambda_{1}^{g}(\Omega)$. Then there exist $\alpha_1,\alpha_2 \in [\alpha,1]$ such that 
$$\lambda_1^g(\Omega)\ge \lambda_1^{\chi}(\Omega),$$
where $\chi=\alpha_1\chi_{\Omega^+}+\alpha_2\chi_{\mathbb R^d\setminus \Omega^+}\in L^\infty_\alpha$.
Equality holds if and only if $u$ is a twisted eigenfunction corresponding to $\lambda_{1}^{\chi}(\Omega)$.

If moreover $\mathcal R(u^+)\le \mathcal R(u^-)$, then 
$$\lambda_1^g(\Omega)\ge \lambda_1^{\chi_\alpha}(\Omega),$$
where $\chi_\alpha= \alpha \chi_{\Omega^+}+\chi_{\Omega^-}$.
Equality holds only if either $g=\chi_\alpha$ a.e. in $\Omega$ or $\mathcal R(u^+)= \mathcal R(u^-)$.
\end{prop}
\begin{proof}
Let $\alpha_1,\alpha_2 \in [\alpha,1]$ be such that $\| u^+g\|_1=\alpha_1\|u^+\|_1$ and $\| u^-g\|_1=\alpha_2\|u^-\|_1$, and define $\chi=\alpha_1\chi_{\Omega^+}+\alpha_2\chi_{\mathbb R^d\setminus\Omega^+}\in L^\infty_\alpha$. Then $u \in H^1_{0,\chi}(\Omega)$ and it can be used in \eqref{lambda} to obtain
\begin{equation*}
\lambda_1^{\chi}(\Omega)\le \frac{\int_{\oo^+} |\nabla u^+(x)|^2  dx+\int_{\oo^-} |\nabla u^-(x)|^2  dx}{\int_{\oo^+}  u^+(x)^2  dx+\int_{\oo^-}  u^-(x)^2  dx}= \lambda_1^{g}(\Omega),
\end{equation*}
which proves the first inequality.
Now, assume $\mathcal R(u^+)\le \mathcal R(u^-)$ and define the linear combination $v$ as in \eqref{e.convex} with $v_1=u^+$, $v_2=u^-$, and $g=\chi_\alpha$ there. Notice that, by the fact that $\alpha\le g\le 1$ a.e. and the orthogonality constraint
\[
\alpha\int_{\Omega_+} u^+dx=\int_{\Omega_+} u^+gdx=\int_{\Omega_-} u^-gdx\le\int_{\Omega_-} u^-dx,
\]
which implies that $\gamma$ as defined in  \eqref{e.convex} is such that $\gamma<1$.
By Lemma~\ref{l.convex} the function $v\in  H^1_{0,\chi_\alpha}(\Omega)$ and it can be used in \eqref{lambda} to obtain
\begin{equation*}
\lambda_1^{\chi_\alpha}(\Omega)\le \frac{\int_{\oo^+} |\nabla u^+(x)|^2  dx+\gamma^2\int_{\oo^-} |\nabla u^-(x)|^2  dx}{\int_{\oo^+}  u^+(x)^2  dx+\gamma^2\int_{\oo^-}  u^-(x)^2  dx}=Q(\gamma^2),
\end{equation*}
with $Q$ as in Lemma~\ref{l.algebraic} and $a_1/a_2=\mathcal R(u^-)\ge\mathcal R(u^+)=b_1/b_2$ by hypothesis, so that 
\begin{equation*}
\lambda_1^{\chi_\alpha}(\Omega)\le Q(\gamma^2)\le Q(1)=\lambda_1^g(\Omega).
\end{equation*}
If $\lambda_1^{\chi_\alpha}(\Omega)=\lambda_1^g(\Omega)$ then all the previous inequalities 
are equalities, in particular $Q(\gamma^2)=Q(1)$ means $\mathcal R(u^+)=\mathcal R(u^-)$ or $\gamma=1$, namely $g=\chi_\alpha$ a.e. in $\Omega$. 
\end{proof}

\begin{oss}
Notice that the equality condition in Proposition~\ref{p.monotonicity2} is only a \emph{necessary condition} but not a sufficient one. However, if $\mathcal R(u^+)<\mathcal R(u^-)$, then the equality $\lambda_1^g(\Omega)= \lambda_1^{\chi_\alpha}(\Omega)$ holds if and only if $g = \chi_\alpha$ a.e. in $\Omega$.
\end{oss}

\subsection{Estimates also in terms of the domain $\Omega$}

As a consequence of the results in the previous sections we now focus on the shape dependence. 

\begin{lem}\label{p.ex1}
If $\Omega=B_+\cup B_-$ where $B_+$ and $B_-$  are disjoint balls with $|B_+|\ge|B_-|$, and $u$ is a twisted eigenfunction corresponding to $\lambda_{1}^{\chi_\alpha}(B_{+} \cup B_{-})$ with $\chi_\alpha=\alpha\chi_{B_+}+\chi_{\mathbb R^d\setminus B_+}$,  then 
\begin{equation*}
    |B_+|^\frac{2}{d}\lambda_{1}^{\chi_\alpha}(B_{+} \cup B_{-})\le{|B|^\frac{2}{d}}\lambda_2(B),
\end{equation*}
where $B$ is any ball in $\mathbb R^d$. 
\end{lem}
\begin{proof}
By testing \eqref{lambda} with the Dirichlet eigenfunction $u_2$ corresponding to $\lambda_2(B_+)$, which clearly belongs to $H^1_{0,\chi_\alpha}(B_+)$, it follows that $\lambda_{1}^{\chi_\alpha}(B_{+} \cup B_{-})\le \lambda_2(B_+)$. The scaling property \eqref{scaling} for Dirichlet eigenvalues provides the bound.
\end{proof}

\begin{oss}
In the case of equality, the twisted eigenvalue is not simple: there exist at least $d$ orthogonal twisted eigenfunctions, which are the ones corresponding to the second Dirichlet eigenvalue on a ball $B$, see \cite[Example 3]{bb}. 
\end{oss}

By testing with a suitable linear combination of the first Dirichlet eigenfunctions on two disjoint balls we also have an upper bound for twisted eigenvalues with orthogonality constraint of bang-bang type.

\begin{prop}\label{p.upper}
If $\Omega=B_+\cup B_-$ where $B_+$ and $B_-$ are two disjoint balls with $|B_+|\ge|B_-|$, and $\chi_\alpha=\alpha\chi_{B_+}+\chi_{\mathbb R^d\setminus B_+}$, then 
\begin{equation*}
\frac{|B_+|^{-1}+\alpha^2|B_-|^{-1}}{|B_+|^{-1-\frac2d}+\alpha^2 |B_-|^{-1-\frac2d}}\lambda_{1}^{\chi_\alpha}(B_{+} \cup B_{-})\le  |B|^\frac{2}{d}\lambda_1(B),  
\end{equation*}
where $B$ is any ball in $\mathbb R^d$.
\end{prop}
\begin{proof}
Let $c$, $c_+$ and $c_-$ be the centers of the balls $B$, $B_+$ and $B_-$. Consider the first Dirichlet eigenfunction $u_1$ of $B$ for some ball $B$ in $\mathbb R^d$ (translate it) and scale it by the factor $a_\pm:=(|B_\pm|/|B|)^{1/d}$ to belong to  $H^1_0(B_\pm)$, namely define $u_\pm(x):=u_1((x-c_\pm +c)/a_\pm)$ for every $x \in B_\pm$. Now, the function $v:=|B_-|u_+-\alpha |B_+|u_-$ satisfies 
\[
\begin{split}
\int_{B_+\cup B_-} v\chi_\alpha dx&=\alpha|B_-|\int_{B_+} u_+dx-\alpha|B_+|\int_{B_-} u_-dx\\&=\alpha|B_-|a_+^d\int_B u_1dx-\alpha|B_+|a_-^d\int_B u_1dx=0
\end{split}
\]
and thus, since $v\in H^1_{0,\chi_\alpha}(B_+\cup B_-)$, one may test \eqref{lambda} with $v$ to obtain
\[
\begin{split}
  \lambda_{1}^{\chi_\alpha}(B_{+} \cup B_{-})&\le   
\frac{|B_-|^2\int_{B_{+}} |\nabla u_+|^2dx  +\alpha^2 |B_+|^2\int_{B_{-}} |\nabla u_-|^2dx }{|B_-|^2\int_{B_{+}} u_+^2 dx+\alpha^2|B_+|^2 \int_{B_{-}}u_-^2dx}\\
&=\frac{|B_-|^2a_+^{d-2}\int_{B} |\nabla u_{1}|^2dx+\alpha^2|B_+|^2a_-^{d-2}\int_{B} |\nabla u_{1}|^2dx}{|B_-|^2a_+^d\int_{B} u_{1}^2dx+\alpha^2|B_+|^2a_-^d\int_{B} u_{1}^2dx}\\
&=\bigg(\frac{|B_-|^2a_+^{d-2}+\alpha^2|B_+|^2a_-^{d-2}}{|B_-|^2a_+^d+\alpha^2|B_+|^2a_-^d}\bigg)\frac{\int_{B} |\nabla u_{1}|^2dx}{\int_{B} u_{1}^2dx},
\end{split}
\]
where the first equality holds by (a translation and) a change of variable with $a_\pm$. Replacing the expressions for $a_\pm$ yields the desired bound.
\end{proof}

\begin{cor}\label{cor.upper}
If $\alpha<1$ there exist two disjoint balls $B_+, B_-$, $|B_+|> |B_-|$ and     \begin{equation*}
      |B_{+} \cup B_{-}|^\frac{2}{d}\lambda_{1}^{\chi_\alpha}(B_{+} \cup B_{-}) < |B^\wedge \cup B^\vee|^\frac{2}{d}\lambda_2(B^\wedge \cup B^\vee),
    \end{equation*}
where $B^\wedge, B^\vee$ are any disjoint balls of equal measure. 
\end{cor}
\begin{proof}
By denoting $m=|B_-|/|B_+|$, the bound in Proposition~\ref{p.upper} reads as follows 
\[
 |B_+\cup B_-|^\frac{2}{d}\lambda_{1}^{\chi_\alpha}(B_{+} \cup B_{-})\le  \frac{m^{1+\frac2d}+\alpha^2 }{m+\alpha^2}\bigg(1+\frac{1}{m}\bigg)^\frac{2}{d}|B|^\frac{2}{d}\lambda_1(B)=:f(m).  
 \]
The function $f$ just defined is continuously differentiable in $(0,1)$ with 
 $$f(1)= |B^\wedge \cup B^\vee|^\frac{2}{d}\lambda_1(B^\wedge \cup B^\vee) \quad \text{and} \quad \lim_{m \rightarrow 1^-}f'(m)=\frac{4^{\frac1d}(1-\alpha^2)}{d(1+\alpha^2)}|B|^\frac{2}{d}\lambda_1(B)>0.$$
Thus, there exists $m<1$, sufficiently close to $1$, such that $f(m)<f(1)$.
\end{proof}

\begin{figure}[t]
\includegraphics[width=4.1cm]{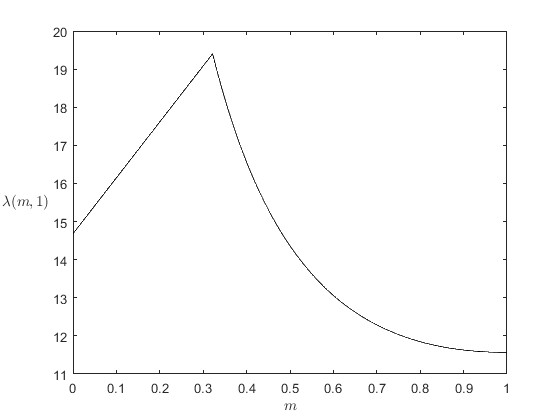}
\includegraphics[width=4.1cm]{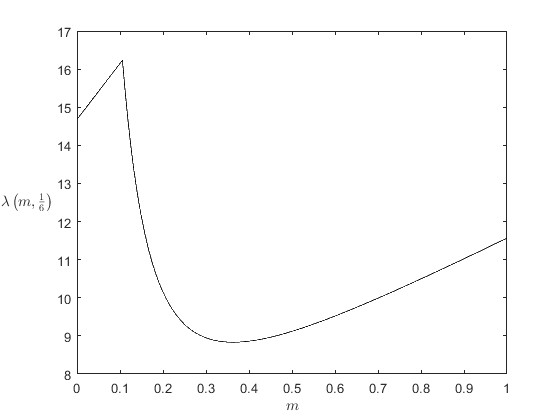}
\includegraphics[width=4.1cm]{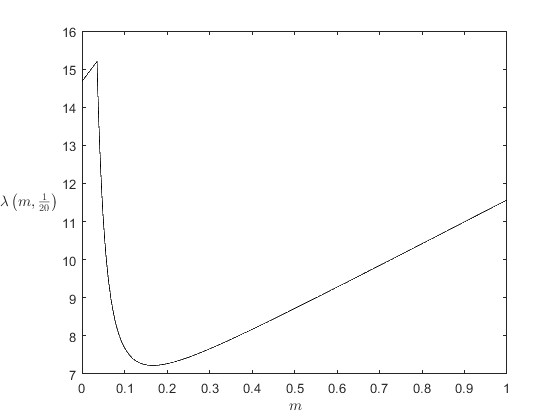}
        \caption{Plots (for $d=2$) of $\lambda(m,\alpha)= |B_+\cup B_-|^{\frac{2}{d}}\lambda_1^{\chi_\alpha}(B_+\cup B_-)$ in terms of the variable $m=|B_-|/|B_+|$, at fixed $\alpha=1,1/6,1/20$ (continuous lines) with their upper bounds in Proposition~\ref{p.upper} (dashed lines).}\label{fig.mlambda}
\end{figure}

Lemma~\ref{p.ex1} with Corollary~\ref{cor.upper} provide bounds on the value $\lambda_1^{\chi_\alpha}(B_+\cup B_-)$ in terms of the ratio $m=|B_-|/|B_+|$; see Figure~\ref{fig.mlambda} for some of these curves.
The following theorem is crucial for proving the main results of the paper: this enables the reduction of the minimization problem to a family of pairs of disjoint balls and bang-bang functions.

\begin{thm}\label{t.rearrangement}
Let $\Omega$ be such that one of the
following conditions holds:
\begin{itemize}
    \item[(i)] $\alpha \in (0,1)$ and $|\Omega|^\frac{2}{d}\lambda_1^g(\Omega)< |B^\wedge \cup B^\vee|^\frac{2}{d}\lambda_2(B^\wedge \cup B^\vee)$;
    \item[(ii)] $\alpha=1$ and $|\Omega|^\frac{2}{d}\lambda_1^g(\Omega)< |B^\wedge|^\frac{2}{d}\lambda_2(B^\wedge)$;
\end{itemize}
where $B^\wedge,B^\vee$ denote disjoint balls of equal measure and $g\in L^\infty_\alpha$. Then there exist $B_{+}$, $B_{-}$ disjoint balls with $|B_{+}| \ge |B_{-}|$ such that
    \begin{equation}\label{inequality.rearrangement}
        |\Omega|^\frac2d\lambda_{1}^{g}(\oo) \ge |B_+\cup B_-|^\frac2d \lambda_{1}^{\chi_\alpha}(B_{+} \cup B_{-}) ,
    \end{equation}
    where $\chi_\alpha = \alpha\chi_{B_{+}} + \chi_{\mathbb R^d\setminus B_+}$ and $B_+,B_-$ are the nodal domains of the twisted eigenfunction corresponding to $\lambda_1^{\chi_\alpha}(B_{+} \cup B_{-})$.
    Equality holds in \eqref{inequality.rearrangement}  if and only if $\oo=B_+ \cup B_-$ and $g=\chi_\alpha=\alpha \chi_{B_+}+\chi_{\mathbb R^d\setminus B_+}$ a.e. in $\Omega$.
\end{thm}
\begin{proof}
By Remark~\ref{r.atleast} there exists a twisted eigenfunction corresponding to $\lambda_1^g(\Omega)$ that changes sign in $\Omega$. By the first statement in Proposition~\ref{p.monotonicity2} there exist $\alpha_1,\alpha_2 \in [\alpha,1]$ such that 
\begin{equation}\label{dis.function}
  \lambda_1^g(\Omega)\ge \lambda_1^{\chi}(\Omega),
\end{equation}
where $\chi=\alpha_1\chi_{\Omega^+}+\alpha_2\chi_{\Omega^-}\in L^\infty_\alpha$.
By using Remark~\ref{r.atleast} again, there exists a twisted eigenfunction $u=u^+-u^-$ corresponding to $\lambda_1^{\chi}(\Omega)$ that changes sign in $\Omega$, which, without loss of generality, satisfies $|\Omega^+|\ge |\Omega^-|$. 
Then consider the \emph{(spherically) symmetric decreasing rearrangements}  (the so-called \emph{Schwarz symmetrizations}) $u_+^*$ and $u_-^*$ of $u^{+}$ and $u^{-}$, respectively. From \cite[Theorems~2.1.2 and 2.1.3]{h}, it is well-known that this rearrangement preserves the $L^p$-norms of a function, while decreasing the $L^p$-norm of its gradient (the so-called \emph{P\'olya-Szeg\H{o} inequality}). Therefore,
\begin{equation}\label{dis.rearrangement}
\begin{split}
    \lambda_{1}^{\chi}(\Omega)
    &=\dfrac{\int_{\oo^+} |\nabla u^{+}(x)|^2  dx + \int_{\oo^-} |\nabla u^{-}(x)|^2  dx}{\int_{\oo^+} u^{+}(x)^2  dx+  \int_{\oo^-} u^{-}(x)^2  dx} \\
    &\ge \dfrac{\int_{\Omega_+^*} |\nabla u_+^{*}(x)|^2  dx+ \int_{\Omega_-^*} |\nabla u_{-}^*(x)|^2  dx}{\int_{\Omega_+^*} u_{+}^*(x)^2  dx+  \int_{\Omega_-^*} u_{-}^*(x)^2  dx},
\end{split}
\end{equation}
where $\Omega_+^*$ and $\Omega_-^*$ are two disjoint balls such that $|\Omega_+^*|=|\Omega^+|$, $|\Omega_-^*|=|\Omega^-|$, $|\Omega_+^*|\ge |\Omega_-^*|$, and eventually $|\Omega_+^*\cup \Omega_-^*|\le |\Omega|$.  
The function $u^*:= u^*_+- u^*_- \in H^1_0(\Omega_+^* \cup \Omega_-^*)$ satisfies
\[
\begin{split}
\int_{\Omega_+^*\cup\Omega_-^*}\!\!\! u^*\chi dx=\alpha_1\!\int_{\Omega_+^*} \!\!u^*_+ dx-\alpha_2\!\int_{\Omega_-^*} \!\!u^*_- dx
&=\alpha_1\!\int_{\Omega^+} \!\!u^+ dx-\alpha_2\!\int_{\Omega^-} \!\! u^- dx\\&=\!\int_{\Omega^+\cup\Omega^-} \!\!\!u\chi dx=\!0,
\end{split}
\]
hence $u^*\in H^1_{0,\chi}(\Omega_+^*\cup \Omega_-^*)$ is admissible in \eqref{lambda} and
\begin{equation}\label{dis.eigenfunction}
    \dfrac{\int_{\Omega_+^*} |\nabla u_+^{*}(x)|^2  dx+ \int_{\Omega_-^*} |\nabla u_{-}^*(x)|^2  dx}{\int_{\Omega_+^*} u_{+}^*(x)^2  dx+  \int_{\Omega_-^*} u_{-}^*(x)^2  dx}
    \ge \lambda_{1}^{\chi}(\Omega_+^*\cup \Omega_-^*) .
\end{equation}
By 
Proposition~\ref{p.rayleigh} and the second statement in  Proposition~\ref{p.monotonicity2} we obtain
\begin{equation}\label{prima}
    \lambda_{1}^{\chi}(\Omega_+^*\cup \Omega_-^*) \ge \lambda_{1}^{\chi_\alpha}(\Omega_+^*\cup \Omega_-^*),
\end{equation}
 where $\chi_\alpha = \alpha\chi_{\Omega_+^*} + \chi_{\Omega_-^*}$. 
Overall, we obtain
\begin{equation}\label{dis.measure}
    |\Omega|^{\frac{2}{d}}\lambda_1^g(\Omega)\ge |\Omega|^{\frac{2}{d}}\lambda_1^{\chi_\alpha}(\Omega_+^*\cup \Omega_-^*)\ge|\Omega_+^*\cup \Omega_-^*|^{\frac{2}{d}}\lambda_1^{\chi_\alpha}(\Omega_+^*\cup \Omega_-^*),
\end{equation}
and \eqref{inequality.rearrangement} follows. 

For the rigidity, the ``if'' part is trivial. Let us prove the ``only if'' part, assuming all the inequalities above to be equalities so that from the hypothesis and \eqref{HKS}
\begin{equation*}
|\Omega_+^*\cup \Omega_-^*|^{\frac{2}{d}}\lambda_1^{\chi_\alpha}(\Omega_+^*\cup \Omega_-^*)=|\Omega_+^*\cup \Omega_-^*|^{\frac{2}{d}}\lambda_1^{\chi}(\Omega_+^*\cup \Omega_-^*)=|\Omega|^{\frac{2}{d}}\lambda_1^g(\Omega)<|B^\wedge|^\frac{2}{d}\lambda_2(B^\wedge)
\end{equation*}
and we can apply Theorem~\ref{simpleCourant} and Lemma~\ref{l.reg} item $(b)$ to infer that $\lambda_1^{\chi_\alpha}(\Omega_+^*\cup \Omega_-^*), \lambda_1^{\chi}(\Omega_+^*\cup \Omega_-^*)$ are simple and their unique (up to scalings) corresponding eigenfunctions  $v=v^+-v^-,w=w^+-w^-$ are given by the difference of two strictly positive analytic functions $v^+,w^+$ and $v^-,w^-$ defined in $\Omega_+^*$ and $ \Omega_-^*$, respectively. Notice that $v$ is also an eigenfunction corresponding to $\lambda_{1}^{\hat g}(\hat \Omega)$, for every $\hat \Omega$ that coincides with $\Omega_+^*\cup \Omega_-^*$ up to a set of zero capacity and $\hat g=\chi_\alpha$ a.e. in $\hat \Omega$. When $\alpha=1$, \eqref{prima} and \eqref{dis.function} reduce to equalities, while when $\alpha<1$, the equality in \eqref{prima} by Proposition~\ref{p.monotonicity2}, either $\chi=\chi_\alpha$ or $\mathcal R(w^+)=\mathcal R(w^-)$. The latter case $\mathcal R(w^+)=\mathcal R(w^-)$ is not possible, since by Proposition~\ref{p.rayleigh}, it follows that $|\Omega_+^*|=|\Omega_-^*|$ and by \eqref{eq.ind} we obtain
$$|\Omega_+^*\cup \Omega_-^*|^\frac{2}{d}\lambda_1^{\chi}(\Omega_+^*\cup \Omega_-^*)=|\Omega_+^*\cup \Omega_-^*|^\frac{2}{d}\lambda_2(\Omega_+^*\cup \Omega_-^*)= |B^\wedge \cup B^\vee|^\frac{2}{d}\lambda_2(B^\wedge \cup B^\vee),$$
which contradicts the hypothesis on $|\Omega|^{\frac2d}\lambda_1^g(\Omega)$. Hence, for $\alpha\in (0,1]$ we have $\chi=\chi_\alpha$ a.e. and $v=w$ (up to scalings). From the equality in \eqref{dis.eigenfunction} and the simplicity of $\lambda_1^{\chi_\alpha}(\Omega_+^*\cup \Omega_-^*)$ we also deduce $w=u^*$ (up to scalings).

Moreover, from the equality in \eqref{dis.rearrangement} and since $u^*=w$ is non vanishing and analytic, by \cite[Proposition 1]{mi} we can apply \cite[Theorem~1.1]{bz}, to deduce that (up to translations) $u^+=u_+^*$ and $u^-=u_-^*$ a.e. in $\R^d$. 
In particular, since $u$ and $u^*$ are analytic $u^+=u_+^*$ and $u^-=u_-^*$ (hence a twisted eigenfunction $u=u^*$, i.e, is radially symmetric and strictly decreasing in each ball), with $\Omega^+=\Omega_+^*$ and $\Omega^-=\Omega_-^*$ two balls. From the equality in \eqref{dis.measure}, $\Omega_+^*\cup \Omega_-^* \subset \Omega$ and $\Omega$ open we have $\Omega \subset \overline{\Omega_+^*\cup \Omega_-^*}$, and thus $\Omega = \Omega_+^*\cup \Omega_-^*$. From the equality in \eqref{dis.function}, by Proposition~\ref{p.monotonicity2}, there exists $\tilde u$ eigenfunction  corresponding to $\lambda_1^g(\Omega_+^*\cup \Omega_-^*)$ and $c \in \R \setminus \{0\}$ such that $cu=\tilde u$. Let $\tilde \Omega \subset \Omega_+^*\cup \Omega_-^*$ be the set such that $g \ne \chi_\alpha$ in $\tilde \Omega$. Since $g\in L^\infty_\alpha$ we have that $g>\alpha$ in $\tilde \Omega \cap \Omega_+^*$ or $g<1$ in $\tilde \Omega \cap \Omega_-^*$, this yields
\begin{equation*}
    \begin{split}
    \int_{\Omega_+^*\cup \Omega_-^*} \!\!\! \tilde u g \, dx
&=c\int_{\Omega_+^*\cup \Omega_-^*} \!\!\! u g \, dx \\
&=c \bigg(\int_{\Omega_+^*\cup \Omega_-^*} \!\!\! u \chi_\alpha \, dx + \int_{\tilde \Omega \cap \Omega_+^*} \!\!\! u^+ (g-\alpha) \, dx + \int_{\tilde \Omega \cap \Omega_-^*} \!\!\! u^- (1-g) \, dx\bigg).
\end{split}
\end{equation*}
Since $u \in H^1_{0,\chi_\alpha}(\Omega_+^*\cup \Omega_-^*)$, $\tilde u \in H^1_{0,g}(\Omega_+^*\cup \Omega_-^*)$, $u^+>0$ in $\Omega_+^*$, $u^->0$ in $\Omega_-^*$ and $g\in L^\infty_\alpha$ we obtain $|\tilde \Omega|=0$. The definition of $\tilde \Omega$ yields $g =\chi_\alpha$ a.e. in $\Omega_+^*\cup \Omega_-^*$ and this concludes the proof.
\end{proof} 

\section{Proof of Theorem~\ref{t.main} and Corollary~\ref{cor.FH}}\label{sec.4}

\begin{proof}[Proof of Theorem~\ref{t.main}]
Fix $0< \alpha \le 1$ and, without loss of generality, we can assume $B_+$, $B_-$ and $\chi_\alpha$ be as in Appendix~\ref{appPoho}. By combining Theorem~\ref{t.rearrangement} with the scale invariance \eqref{scaling} and the (partial) translation invariance of $|B_+\cup B_-|^\frac{2}{d}\lambda_1^{\chi_\alpha}(B_+\cup B_-)$, problem \eqref{problem} is equivalent to the minimization of twisted eigenvalues over the class made up of two disjoint \emph{non-empty} balls of total measure equal to $\delta_d$, contained in a fixed domain $D\subset \mathbb R^d$ (the \emph{box}), with orthogonality constraint of bang-bang type, namely one reduces to study
\begin{equation}\label{problembis}
\lambda(\alpha)=\delta_d^{2/d}\min_{\Omega\in\mathcal O_1(D)} \lambda_1^{\chi_\alpha}(\Omega)
\end{equation}
with $$\mathcal O_1(D):=\{B_{+} \cup B_-\subset D: \text{ $B_+, B_-$ disjoint balls, $|B_+|\ge |B_-|$,  $|B_+|+|B_-|=\delta_d$}\}$$
and the bang-bang function $\chi_\alpha=\alpha\chi_{B_+}+\chi_{\mathbb R^d\setminus B_+}$.
In other words, proving Theorem~\ref{t.main} is equivalent to establishing its properties for the minimizers of \eqref{problembis}. The proof is divided into four steps.  

\smallskip
\emph{Step 1 (existence of a minimizer).} We claim that there exists a pair of {disjoint} \emph{non-empty} balls minimizing \eqref{problembis}.
The existence of such a minimizer is standard and follows the Direct Methods of the Calculus of Variations. Let $\{\Omega_n=B_+^n\cup B_-^n\}_n\subset \mathcal O_1(D)$ be a minimizing sequence for \eqref{problembis} and define $g_n:=\alpha \chi_{B_{+}^n}+\chi_{\mathbb R^d\setminus B_+^n}$. 
By compactness -- since we are working with pairs of balls of fixed volume -- up to subsequences, the sequence $\{\Omega_n\}_n$ converges (e.g., in terms of the radii of the balls) to some limit set $\Omega=B_+\cup B_-$. Note that the limit may degenerate into a single ball of measure $\delta_d$, in the case $B_-=\emptyset$. Moreover, $g_n\to\chi_\alpha=\alpha \chi_{B_+}+\chi_{B_-}$ a.e. in $D$. Now, let $u_n$ be a twisted eigenfunction corresponding to $\lambda_{1}^{g_n}(\Omega_n)$ normalized in $L^2(\Omega_n)$. By minimality, the sequence $\{u_n\}_n$ is uniformly bounded in $H^1_0(D)$. Therefore, up to extracting further subsequences, we may assume that $\{u_n\}_n$ converges weakly in $H^1_0(D)$, strongly in $L^2(D)$, and a.e. in $D$ to some function $u\in H^1_0(D)$ such that 
    \begin{equation*}
        \int_D \!u(x)^2 dx 
        \!=\!\!\!\lim_{n \rightarrow +\infty} \int_D\! u_n(x)^2 dx
        =1,
        \int_D\! u(x)\chi_\alpha(x) dx
        \!=\!\!\!\lim_{n \rightarrow +\infty} \int_D\! u_n(x)g_n(x) dx
        =0.
    \end{equation*}
Since $u\equiv 0$ a.e in $D\setminus \Omega$ and $\Omega$ is a Lipschitz set, it follows from \cite[Remarks 3.4.8 and 3.4.9]{hp} that $u \in H^1_0(\oo)$. Combined with the previous equality, this implies that $u \in H^1_{0,g}(\oo)$. By lower semicontinuity of the $H^1$-norm w.r.t. the weak convergence in $H^1$
    \begin{equation*}
        \lambda_{1}^{\chi_\alpha}(\Omega)
        \le \int_D |\nabla u(x)|^2  dx
        \le \liminf_{n \to \infty} \int_D |\nabla u_n(x)|^2  dx
        =\frac{\lambda(\alpha)}{\delta_d^{2/d}},
    \end{equation*}
    hence $\Omega$ is a minimizer for \eqref{problembis} and $u$ is a twisted eigenfunction corresponding to $\lambda_{1}^{\chi_\alpha}(\Omega)$.
    By \eqref{eq.ind} and the rigidity statement of \eqref{HKS} we have 
    \begin{equation}\label{sotto}
        \lambda(\alpha)\le |B^\wedge \cup B^\vee|^{\frac{2}{d}} \lambda_1^{\chi_\alpha}(B^\wedge \cup B^\vee)=|B^\wedge \cup B^\vee|^{\frac{2}{d}} \lambda_2(B^\wedge \cup B^\vee)< |B|^\frac{2}{d}\lambda_2(B),
    \end{equation}
    for any ball $B$ in $\mathbb R^d$. This shows that the minimizer $\Omega\in \mathcal O_1(D)$, thus by Theorem~\ref{simpleCourant} we obtain that every twisted eigenfunction corresponding to $\lambda_{1}^{\chi_\alpha}(\Omega)$ has nodal domains $\Omega^+=B_+$ and $\Omega^-=B_-$. 

\smallskip
\emph{Step 2 (the optimality condition).}    
Let $\Omega_\alpha=B_{+} \cup B_{-}\in \mathcal O_1(D)$ be a minimizer of \eqref{problembis}. By \eqref{sotto} we can apply Proposition~\ref{p.shapederivative} to $\Omega_\alpha$ and by minimality of $\Omega_\alpha$ we may infer the existence of a constant $\mu$ (the Lagrange multiplier for the volume constrained problem \eqref{problembis}, see \cite[Section~5.9.3, p.~243]{hp}) such that
\begin{equation}\label{eq.derivative}
    \int_{{\partial\Omega_\alpha}} \left(|\nabla u|^2 -\mu\right)\,W \cdot \nu \, d\mathcal{H}^{d-1}=0,
\end{equation}
     for every smooth vector field $W$ from $\R^d$ to $\R^d$ and, by the arbitrariness of $W$,
    \begin{equation}\label{oc}
     |\nabla u(x)|^2=\mu \quad \text{for a.e. $x\in \partial\Omega_\alpha$}.
     \end{equation}
By taking $W=X$ in \eqref{eq.derivative}, and then by using Proposition~\ref{p.pohozaev}, it follows that
\begin{equation*}
    2 \lambda_{1}^{\chi_\alpha}(\Omega_\alpha)
    =\int_{\partial\oo_\alpha} \!\! |\nabla u|^2 \,X \cdot \nu   d\mathcal{H}^{d-1}
    = \mu\int_{\partial\oo_\alpha} \!\! X \cdot \nu d\mathcal{H}^{d-1}
    = \mu\int_{\oo_\alpha} \!\! \textup{div}(X)  dx
    =\mu d |\Omega_\alpha|  \, ,
\end{equation*}
which gives the explicit value of the Lagrange multiplier in \eqref{oc}, and \eqref{eq.optimality} follows.    

\smallskip
\emph{Step 3 (bounds on measure ratio and eigenvalues of optimal shapes).} Let $\Omega_\alpha=B_+\cup B_-\in\mathcal O_1(D)$ be a minimizer of \eqref{problembis} and $u$ be a twisted eigenfunction corresponding to $\lambda_{1}^{\chi_\alpha}(\Omega_\alpha)$.
By plugging the expression \eqref{eq.xi} into Corollary~\ref{cor.xi} yields
   \begin{equation*}
       0\ge  -\int_{B_+\cup B_-} \Delta u(x) \chi_\alpha(x)  dx=-\alpha\int_{B_+} \Delta u^+(x) dx +\int_{B_-} \Delta u^-(x) dx.
    \end{equation*}
The divergence theorem with the optimality condition \eqref{oc} allows to compute the integrals in the right-hand side of the previous inequality and rewrite it as follows
\begin{equation*}
 0\ge \sqrt{\mu}\left(\alpha |B_+|^\frac{d-1}{d}-|B_-|^\frac{d-1}{d}\right),
    \end{equation*}
where we also used the fact that for every ball $B$ its perimeter $P(B)$ satisfies $P(B)=d{\delta_d^{1/d}}|B|^\frac{d-1}{d}$. By solving this inequality with respect to $m(\alpha)$ the first bound in \eqref{dis.radius} follows. 
By applying \eqref{interwinedTD}, \cite[Remark 1.2.4]{h} and \eqref{scaling} we obtain
    \begin{equation*}
    \begin{split}
\lambda(\alpha) 
        &= |B_+^\alpha \cup B_-^\alpha|^\frac{2}{d}\lambda_{1}^{\chi_\alpha}(B_+^\alpha \cup B_-^\alpha)
        \ge |B_+^\alpha \cup B_-^\alpha|^\frac{2}{d}\lambda_{1}(B_+^\alpha \cup B_-^\alpha)\\
        &= |B_+^\alpha \cup B_-^\alpha|^\frac{2}{d}\lambda_{1}(B_+^\alpha)
        =(1+m(\alpha))^\frac{2}{d}|B|^\frac{2}{d}\lambda_1(B),
        \end{split}
    \end{equation*}
that proves the second bound in \eqref{dis.radius}.

\smallskip
\emph{Step 4 (uniqueness of the minimizer).} We proceed by contradiction, assuming the existence of two pairs of disjoint balls $B_{+} \cup B_{-}$ of radii $R_+,R_-$ and $B_{\oplus} \cup B_{\ominus}$ of radii $R_{\oplus}, R_{\ominus}$ with $R_{+}<R_{\oplus}<1$ and such that
    \begin{equation*}
        \lambda_{1}^{\chi_\alpha}(B_{+} \cup B_{-})
        =\lambda_{1}^{\overline{\chi}_\alpha}(B_{\oplus} \cup B_{\ominus})
        =\lambda_{1}^{\chi_\alpha}(\Omega_\alpha)
        =\lambda,
    \end{equation*}
     with $\chi_\alpha=\alpha\chi_{B_+}+\chi_{\mathbb R^d\setminus B_+}$ and $\overline\chi_\alpha=\alpha\chi_{B_\oplus}+\chi_{\mathbb R^d\setminus B_\oplus}$. Notice that the volume constraint in \eqref{problembis} implies $R_+^d+R_-^d=1$ and $R_\oplus^d+R_\ominus^d=1$.
If $u=u^+-u^-$ and  $v=u^{\oplus}-u^{\ominus}$ are twisted eigenfunctions corresponding to $\lambda_1^{\chi_\alpha}(B_{+} \cup B_{-})$ and to $\lambda_1^{\overline\chi_\alpha}(B_{\oplus} \cup B_{\ominus})$ (here $u^{\oplus}$ and $u^{\ominus}$ denote positive and negative part of $v$), then by Theorem~\ref{t.rearrangement} we know they are radially symmetric (with respect to the center of each ball). 
Let $u_\pm$ denote the $1$-dimensional functions describing the radial profiles (centered at the origin) of $u^\pm$, which by
\eqref{system.constant} and the radiality property satisfy the system
\begin{equation}\label{system.ode}
    \begin{cases}
        u_\pm''(r) + \frac{d-1}{r} u_\pm'(r) + \lambda u_\pm(r)=\xi_\pm & \text{for $r \in (0,R_\pm)$}, \\
        u_\pm(R_\pm)=0 \, , \, u_\pm'(0)=0 \, ,
    \end{cases}
\end{equation}
where $\xi_+=-\alpha\xi$, $\xi_-= \xi$ for some $\xi \in \mathbb R$, and in particular $\xi_+=-\alpha \xi_-$. From the optimality condition \eqref{oc} there holds
\begin{equation}\label{uguaglianza}
    \left| \dfrac{u'_{+}(R_{+})}{u'_{-}(R_{-})} \right|
        = \left| \dfrac{u'_{\oplus}(R_{\oplus})}{u'_{\ominus}(R_{\ominus})} \right|=1.
\end{equation}

On the other hand, by Proposition~\ref{ODE} solutions to \eqref{system.ode} read as follows
\begin{equation}\label{explicit_radial}
        u_{\pm}(r)=s_{\pm}\left(r^{1-\frac{d}{2}} J_{\frac{d}{2}-1}(\omega r) - R_{\pm}^{1-\frac{d}{2}} J_{\frac{d}{2}-1}(\omega R_{\pm}) \right) \quad \text{for $r \in (0,R_{\pm})$}
\end{equation}
with $s_{\pm} \in \R$ and $\omega=\sqrt{\lambda}$. By inserting the explicit expressions \eqref{explicit_radial} for $u_\pm$ into the orthogonality condition $\int_{B_{+} \cup B_{-}} u(x)\chi_\alpha(x)  dx=0$, and writing the integral in polar coordinates, it follows that
\begin{equation*}
\begin{split}
   0= \int_{B_+\cup B_-} u(x)\chi_\alpha(x) dx=& \, \alpha s_{+} \Big( d \delta_d \int_0^{R_{+}} 
    r^{\frac{d}{2}} J_{\frac{d}{2}-1}(\omega r) \, dr - \delta_d R_{+}^{\frac{d}{2}+1} J_{\frac{d}{2}-1}(\omega R_{+})
    \Big) \\
    &\! - s_{-} \Big( d \delta_d \int_0^{R_{-}} 
    r^{\frac{d}{2}} J_{\frac{d}{2}-1}(\omega r) \, dr - \delta_d R_{-}^{\frac{d}{2}+1} J_{\frac{d}{2}-1}(\omega R_{-})
    \Big),
\end{split}
\end{equation*}
where $d \delta_d$ is the $(d-1)$-measure of the unit sphere in $\R^d$. Relying on \eqref{bessel.integral} and \eqref{recurrence.relation} we obtain
\begin{equation}\label{ort.condition}
\alpha s_+R_+^{\frac{d}{2}+1}J_{\frac{d}{2}+1}(\omega R_+)-s_-R_-^{\frac{d}{2}+1}J_{\frac{d}{2}+1}(\omega R_-)=0
\end{equation}
and, in particular, we may take
\begin{equation}\label{expression.coefficients}
    s_{+}= R_-^{\frac{d}{2}+1}J_{\frac{d}{2}+1}(\omega R_-)
    \quad \text{ and } \quad
    s_{-}= \alpha R_+^{\frac{d}{2}+1}J_{\frac{d}{2}+1}(\omega R_+).
\end{equation}
By differentiating \eqref{explicit_radial} with respect to $r$, and by using the second relation in \eqref{bessel.derivative}, we obtain
\begin{equation*}
        u'_{\pm}(r) =-s_{\pm} \omega r^{1-\frac{d}{2}} J_{\frac{d}{2}}(\omega r)  \quad
        \text{for every $r \in (0,R_{\pm})$}
\end{equation*}
which combined with \eqref{expression.coefficients} allows to write \emph{the optimality condition} \eqref{uguaglianza}, as an implicit equation relating $\alpha$, $\omega$ and the radii $R_\pm$ 
\begin{equation}\label{optimality}
    \left| \dfrac{u'_{+}(R_{+})}{u'_{-}(R_{-})} \right|
    =\frac{1}{\alpha}\dfrac{\Phi(\omega R_{-})}{\Phi(\omega R_{+})}=1, \quad \text{with  $\Phi(x):=x^{d} \frac{J_{\frac{d}{2}+1}(x)}{J_{\frac{d}{2}}(x)}$}.
\end{equation}
Similarly an equation as the previous one holds for a twisted eigenfunction $v=u^{\oplus}-u^{\ominus}$ corresponding to $\lambda_1^{\chi_\alpha}(B_{\oplus} \cup B_{\ominus})$. From the hypotheses, the scaling-invariance and \eqref{sotto}
\begin{equation*}
    0<\omega R_{\ominus} < \omega R_{-} \le \omega R_{+} <\omega R_{\oplus} \le |\Omega_\alpha|^\frac{1}{d} \sqrt{\lambda_{1}^{\chi_\alpha}(\Omega_\alpha)} < |B|^\frac{1}{d}\sqrt{\lambda_2(B)} =j_{\frac{d}{2},1},
\end{equation*}
and since by Proposition~\ref{p.phi-increasing} the function $\Phi$ is strictly  increasing in  
$(0,j_{\frac{d}{2},1})$, we obtain
\begin{equation*}
    \left| \dfrac{u'_{+}(R_{+})}{u'_{-}(R_{-})} \right|
    =\frac{1}{\alpha} \dfrac{\Phi(\omega R_{-})}{\Phi(\omega R_{+})}
    >\frac{1}{\alpha} \dfrac{\Phi(\omega R_{\ominus})}{\Phi(\omega R_{\oplus})}
    =\left| \dfrac{u'_{\oplus}(R_{\oplus})}{u'_{\ominus}(R_{\ominus})} \right| ,
\end{equation*}
that is in contradiction with \eqref{uguaglianza} and the uniqueness follows.
\end{proof}

\begin{oss}
In dimension $d=2$, for the values of $\alpha$ considered in Figure~\ref{fig.opt} we have
\begin{center}
\begin{tabular}{|c c c|} 
 \hline
 $\alpha$ & $\lambda(\alpha)$ & $m(\alpha)$  \\ [0.2ex]
 \hline \hline
 1 & $\pi\cdot j_{0,1}^2 \approx 36.3368$ & 1  \\ [0.2ex]
 \hline
 1/6 & 27.7534 & 0.3635 \\ [0.2ex]
 \hline
 1/20 & 22.7001 & 0.1664  \\ [0.2ex]
 \hline
\end{tabular}
\end{center}
\end{oss}

\begin{oss}\label{r.orthogonality}
Let $\omega=\sqrt{\lambda}$, where $\lambda$ is defined as in \eqref{system.constant}. By \eqref{relazione} and $\xi_+=-\alpha \xi_-$, a configuration solving of \eqref{system.constant} satisfies
\[
s_+ R_+^{1-\frac{d}{2}} J_{\frac{d}{2}-1}(\omega R_+)=-\alpha s_- R_-^{1-\frac{d}{2}} J_{\frac{d}{2}-1}(\omega R_-),
\]
hence one may also take
\begin{equation*}
    s_{+}= -\alpha R_{-}^{1-\frac{d}{2}}J_{\frac{d}{2}-1}(\omega R_{-})
    \quad \text{ and } \quad
    s_{-}= R_{+}^{1-\frac{d}{2}}J_{\frac{d}{2}-1}(\omega R_{+}).
\end{equation*}
By inserting these coefficients into the \emph{orthogonality condition}
\eqref{ort.condition} we obtain another \emph{implicit equation} relating $\alpha$, $\omega$ and the radii $R_\pm$
\[
 R_-^{\frac{d}{2}+1}J_{\frac{d}{2}+1}(\omega R_-) R_+^{1-\frac{d}{2}}J_{\frac{d}{2}-1}(\omega R_+)
    + \alpha^2 R_+^{\frac{d}{2}+1}J_{\frac{d}{2}+1}(\omega R_+) R_-^{1-\frac{d}{2}}J_{\frac{d}{2}-1}(\omega R_-)=0 .
\]
If $\omega R_+ \notin \{j_{\frac{d}{2}-1,\sigma}\}_{\sigma}$ (and so, by \emph{Porter's theorem} \cite[Section 15.22, p. 480]{w} $\omega R_- \notin \{j_{\frac{d}{2}-1,\sigma}\}_{\sigma}$), namely the twisted eigenfuctions corresponding to
this configuration are not Dirichlet ones, it can be rewritten as
\begin{equation}\label{ortogonality}
\phi(\omega R_-)+\alpha^2 \phi(\omega R_+)=0, \quad \phi(x):=x^{d} \frac{J_{\frac{d}{2}+1}(x)}{J_{\frac{d}{2}-1}(x)}.
\end{equation}
If $\alpha<1$, by Corollary~\ref{cor.upper}, \eqref{sotto} and the equality case in  Proposition~\ref{p.rayleigh}, twisted eigenfunctions are not Dirichlet ones, hence in this case \eqref{ortogonality} holds.
\end{oss}

\begin{proof}[Proof of Corollary~\ref{cor.FH}]
If $\alpha=1$ then Corollary~\ref{cor.FH} follows directly from the \emph{Steps~1, 2, and 3} in the proof of Theorem~\ref{t.main}. Since this three steps do not rely on Bessel functions, by letting $\alpha=1$ in \eqref{dis.radius}, and recalling \eqref{sotto}, we obtain a direct (without the Bessel function) proof of Corollary~\ref{cor.FH}.
\end{proof}

\begin{oss}
The uniqueness of the minimizer proved in Step~4, together with the optimality condition \eqref{oc}, gives another proof of Corollary~\ref{cor.FH} (in this case it uses Bessel functions). Indeed, let $B^\wedge$ and $B^\vee$ be any pair of disjoint balls in $\mathbb R^d$ of equal measure and $u_1$ and $v_1$ be the first Dirichlet eigenfunctions of $B^\wedge$ and $B^\vee$, respectively. The twisted eigenfunction $u_1-v_1$ corresponding to $\lambda_{1}^{T}(B^\wedge \cup B^\vee)$ has constant modulus on $\partial B^\wedge\cup \partial B^\vee$. 
\end{oss}

\begin{oss}
The \emph{optimality condition}  \eqref{optimality} can also be obtained using the Lagrange multiplier methods on the relation \eqref{ortogonality} with a volume constraint, as in \cite[Proposition 4.4]{gl}. This method, alternative to the shape derivative, guarantees a unique minimizer but does not lead to equation \eqref{eq.optimality}. This strategy of proof relies on Bessel functions.
\end{oss}

\section{Proof of Theorem~\ref{t.main2} and Corollary~\ref{c.nonexistence}}\label{sec.5}

\begin{proof}[Proof of Theorem~\ref{t.main2}]
The proof is an application of the implicit function theorem for vector-valued maps. Although Theorem~\ref{t.main} already ensures the existence of the functions $m(\alpha)$ and $\lambda(\alpha)$ in terms of the variable $\alpha$, the use of the implicit function theorem is needed to establish their continuous differentiability. To this end, it is convenient to consider a vectorial-valued map $F$ defined on $\alpha$ and on two \emph{auxiliary} variables: a scaled measure $M=R_+^d$ of the largest ball $B_+$ of an optimal pair and the frequency $\omega=\sqrt{\lambda}$. As in Theorem~\ref{t.main} we consider optimal sets of measure $\delta_d$. These are related to the original variables of the theorem through the identities  $\lambda(\alpha)=\delta_d^{2/d}\omega(\alpha)^2$ and 
\[
m(\alpha)=\frac{|B_-^\alpha|}{|B_+^\alpha|}=\frac{\delta_d-|B_+^\alpha|}{|B_+^\alpha|}=\frac{1}{R_+(\alpha)^d}-1=\frac{1}{M(\alpha)}-1,
\]
where $R_+(\alpha)$ is the radius of the ball $|B_+^\alpha|$. Once the continuous differentiability and the monotonicity are established for these auxiliary variables, the continuous differentiability and the monotonicity of $m$ and $\lambda$ follows directly from the previous identities. The proof is divided into two steps.

\smallskip
\emph{Step 1 (continuous differentiability of the functions).} Consider $F\colon X\subset\mathbb R^3\to \mathbb R^2$ with $X=(0,1)\times ({1}/{2},1)\times (j_{d/2-1,1},j_{d/2,1})$ and $F=(f_1,f_2)$ and $f_1,f_2\colon X\to \mathbb R$ scalar functions defined by
$$
f_1(\alpha, M, \omega):=\phi(\omega (1-M)^\frac{1}{d})+\alpha^2\phi(\omega M^\frac{1}{d}), \quad \text{with 
 $\phi(x):=x^{d} \frac{J_{\frac{d}{2}+1}(x)}{J_{\frac{d}{2}-1}(x)}$}
$$
and 
$$
f_2(\alpha, M, \omega):=\Phi(\omega (1-M)^{\frac1d})-\alpha\Phi(\omega M^{\frac1d}), \quad \text{with  $\Phi(x):=x^{d} \frac{J_{\frac{d}{2}+1}(x)}{J_{\frac{d}{2}}(x)}$}.
$$
Notice that by Remark~\ref{r.orthogonality} and \eqref{optimality} there holds $F(\alpha, M(\alpha),\omega(\alpha))=(0,0)$ for every $\alpha\in(0,1)$. Therefore, we may consider a point $x_0=(\alpha_0, M_0, \omega_0)\in X$ such that $M_0=M(\alpha_0)$ and $\omega_0=\omega(\alpha_0)$, then $F(x_0)=(0, 0)$. Since $f_1(x_0)=0$ by Proposition~\ref{p.phi} $\phi$ is negative at $\omega_0 M_0^{1/d}\in (j_{d/2-1},j_{d/2})$ while positive at $\omega_0 (1-M_0)^{1/d}\in (0,j_{d/2-1})$. By definition $F$ is continuously differentiable in $x_0$ with respect to all the three variables. Therefore, to apply the implicit function theorem to $F$, see for instance \cite{KP}, it remains to show that the Jacobian of $F$ at $x_0$ with respect the variables to be made explicit --namely, $M$ and $\omega$-- is invertible. We compute first the partial derivatives of $f_1$:
       \begin{equation*}
       \begin{split}
       \frac{\partial f_1}{\partial M} (\alpha_0,M_0,\omega_0)&=\frac{\omega_0}{d}\left(-(1-M_0)^{\frac{1}{d}-1}\phi'(\omega_0 (1-M_0)^{\frac{1}{d}})+\alpha_0^2 M_0^{\frac{1}{d}-1}\phi'(\omega_0 M_0^{\frac{1}{d}})\right)
        \end{split}
    \end{equation*}
and 
    \begin{equation*}
         \frac{\partial f_1}{\partial \omega} (\alpha_0,M_0,\omega_0)
        =(1-M_0)^{\frac1d}\phi'(\omega_0 (1-M_0)^{\frac1d})+ \alpha_0^2 M_0^{\frac1d} \phi'(\omega_0 M_0^{\frac1d}).
    \end{equation*} 
Now, for $f_2$ we have   
    \begin{equation*}
        \frac{\partial f_2}{\partial M} (\alpha_0,M_0,\omega_0)
        =-\frac{\omega_0}{d}\left( (1-M_0)^{\frac{1}{d}-1} \Phi'(\omega_0 (1-M_0)^{\frac1d})+\alpha_0 M_0^{\frac{1}{d}-1}\Phi'(\omega_0 M_0^{\frac1d})   \right)
    \end{equation*}
and  
       \begin{equation*}
       \begin{split}
        \frac{\partial f_2}{\partial \omega} (\alpha_0,M_0,\omega_0)
        &=(1-M_0)^{\frac1d}\Phi'(\omega_0 (1-M_0)^{\frac1d})-\alpha_0 M_0^{\frac1d} \Phi'(\omega_0 M_0^{\frac1d}).
        \end{split}
    \end{equation*}
The determinant $\mathrm{det} JF$ of the Jacobian matrix of $F$ (computed with respect to the variables $M$ and $\omega$) evaluated at $x_0$ is given by
\[
\begin{split}
\Big|  \frac{\partial f_1}{\partial M}  \frac{\partial f_2}{\partial \omega} -  \frac{\partial f_1}{\partial \omega}  \frac{\partial f_2}{\partial M}\Big|&\!=\!\frac{1}{d \omega_0}\!\!\left(\!-\frac{x_- \phi'(x_-)}{1-M_0}\!+\!\frac{\alpha_0^2 x_+ \phi'(x_+)}{M_0}\right)\!\!(x_-\Phi'(x_-)-\alpha_0x_+\Phi'(x_+))\\
&+\!\frac{1}{d\omega_0}\left({x_- \phi'(x_-)}+\alpha_0^2x_+ \phi'(x_+)\right)\!\!\left(\frac{x_-\Phi'(x_-)}{1-M_0}+\frac{\alpha_0x_+\Phi'(x_+)}{M_0}\right)\\
&=\frac{\alpha_0 x_-x_+}{d\omega_0M_0(1-M_0)}\left(\phi'(x_-)\Phi'(x_+)+\alpha_0\phi'(x_+)\Phi'(x_-)\right),
\end{split}
\]
 where we denoted by $x_-=\omega_0(1-M_0)^{\frac1d}$ and $x_+=\omega_0 M_0^{\frac1d}$. 
By Propositions~\ref{p.phi} and \ref{p.phi-increasing} we deduce that the determinant $\mathrm{det} JF$ of the Jacobian matrix of $F$ (computed with respect to the variables $M$ and $\omega$) evaluated at $x_0$ is strictly positive, namely the Jacobian matrix is invertible. Therefore, the implicit function theorem \cite[Theorem~3.3.1]{KP} applies and there uniquely exist two functions, that we already denoted as $M(\alpha)$ and $\omega(\alpha)$, that are continuously differentiable for every $\alpha \in (0,1)$.

\smallskip
\emph{Step 2 (monotonicity of the functions).}
Let $0<\alpha_1< \alpha_2<1$ and $\Omega_{\alpha_1},\Omega_{\alpha_2}$ be the sets with measure $\delta_d$ achieving the equality in \eqref{isoperimetric} for $\alpha=\alpha_1,\alpha=\alpha_2$, respectively. By Remark~\ref{r.orthogonality} the eigenfunction $u_{\alpha_2}$ is not a Dirichlet eigenfunction, therefore combining Propositions~\ref{prop.xi} and \ref{p.monotonicity2} we obtain
\begin{equation*}
    \omega(\alpha_2)^2=
    \lambda_{1}^{\chi_{\alpha_2}}(\Omega_{\alpha_2}) 
    > \lambda_{1}^{\chi_{\alpha_1}}(\Omega_{\alpha_2}) 
    \ge \lambda_{1}^{\chi_{\alpha_1}}(\Omega_{\alpha_1}) 
    =\omega(\alpha_1)^2. 
\end{equation*}
Thus the function $\omega$ is strictly increasing in $(0,1)$ and in particular $\omega'(\alpha_0) \ge  0$ for every $\alpha_0 \in (0,1)$. The function $f\colon (0,1) \to \mathbb R$ defined by
\[
\alpha \mapsto f(\alpha)=f_2(\alpha,M(\alpha),\omega(\alpha))
\]
is such that $f(\alpha_0)=0$ for every $\alpha_0\in (0,1)$ and differentiable with respect to $\alpha$ at $\alpha_0$ with
\[
f'(\alpha_0)=\frac{\partial f_2}{\partial \alpha}(\alpha_0, M_0, \omega_0)+M'(\alpha_0)\frac{\partial f_2}{\partial M}(\alpha_0,M_0,\omega_0)+\omega'(\alpha_0)\frac{\partial f_2}{\partial \omega}(\alpha_0,M_0,\omega_0)=0.
\]
Solving this equation in terms of the unknown $M'$ we have that
\[
\begin{split}
M'(\alpha_0)&=\frac{\frac{\partial f_2}{\partial \alpha}(\alpha_0, M_0, \omega_0)+\omega'(\alpha_0)\frac{\partial f_2}{\partial \omega}(\alpha_0,M_0,\omega_0)}{-\frac{\partial f_2}{\partial M}(\alpha_0,M_0,\omega_0)}<0,
\end{split}
\]
since the partial derivative  w.r.t. $\alpha$ is $(\partial f_2/\partial \alpha) (\alpha_0, M_0,\omega_0)= - \Phi(\omega_0 M_0^{\frac{1}{d}})<0$, moreover $(\partial f_2/\partial M) (\alpha_0, M_0,\omega_0)<0$ by the sign of $\Phi'$, see Proposition~\ref{p.phi-increasing}, while
\[
\begin{split}
        \frac{\partial f_2}{\partial \omega} (\alpha_0,M_0,\omega_0)
        &=(1-M_0)^{\frac1d}\Phi'(\omega_0 (1-M_0)^{\frac1d})-\alpha_0 M_0^{\frac1d} \Phi'(\omega_0 M_0^{\frac1d})\\
        &=\Phi(\omega_0 (1-M_0)^{\frac1d})(\Upsilon(\omega_0 (1-M_0)^{\frac1d})- \Upsilon(\omega_0 M_0^{\frac1d}))<0,
        \end{split}
\]
 where we used that $f_2(\alpha_0,M_0,\omega_0)=0$, the positivity of $\Phi$ in Proposition~\ref{p.phi-increasing} with Remark~\ref{oss.notazione} and Proposition~\ref{p.upsilon} to rewrite and deduce the sign of the partial derivative of $f_2$ w.r.t.  $\omega$.
Combining Proposition~\ref{p.espilon} with $\Omega=B$ and \eqref{FK} we obtain $\lim_{\alpha\to 0^+}\lambda(\alpha)=|B|^\frac{2}{d}\lambda_1(B)$, where $B$ is any ball in $\mathbb R^d$. Suppose by contradiction that $\lim_{\alpha\to 0^+}m(\alpha)= \overline m >0$, then by \eqref{interwinedTD} and \cite[Remark 1.2.4]{h} we have
$$|B_+^\alpha  \cup B_-^{\alpha}|^{\frac{2}{d}}\lambda_{1}^{\chi_\alpha}(B_+^{\alpha} \cup B_-^\alpha )
\ge |B_+^\alpha \cup B_-^{\alpha}|^{\frac{2}{d}} \lambda_{1}(B_+^{\alpha})
> (1+\overline m)^{\frac{2}{d}}|B_+^\alpha|^{\frac{2}{d}} \lambda_{1}(B_+^{\alpha}),$$
therefore $\lim_{\alpha\to 0^+}\lambda(\alpha)\ge (1+\overline m)^{\frac{2}{d}}|B|^\frac{2}{d}\lambda_1(B)$, a contradiction. Moreover, by \eqref{dis.radius} in Theorem~\ref{t.main} letting $\alpha\to 1^-$ we deduce $m(1)=1$ and $\lambda(1)= |B^\wedge \cup B^\vee|^\frac{2}{d}\lambda_2(B^\wedge\cup B^\vee)$ and the theorem is proved.
\end{proof}

\begin{proof}[Proof of Corollary~\ref{c.nonexistence}]
Let $(\alpha_n)_n \subset \R$ be a strictly decreasing sequence bounded by $1$ with $\alpha_n \rightarrow \alpha$, as $n \rightarrow \infty$, and consider $g \in L^\infty_\alpha(\mathbb R^d)$ defined by
    $$g=\chi_{B_1} + \sum_{n=1}^\infty \alpha_n \chi_{B_{2^n}\setminus B_{2^{n-1}}},$$
    where $B_r$ is the ball of radius $r>0$ centered at the origin ($g$ is a function that decreases as the spherical shell expands).   
\begin{figure}[H] \centering
\begin{tikzpicture}[scale=0.42]
\draw (0,0) circle (1); \node at
(0,0) {$1$}; \draw (0,0) circle
(2); \node at (1.5,0)
{$\alpha_1$}; \draw (0,0) circle
(4); \node at (3,0)
{$\alpha_2$}; \draw (0,0) circle
(8); \node at (6,0)
{$\alpha_3$}; \node at (9,0)
{$\cdots$}; \end{tikzpicture}
\caption{The function $g\in L^\infty_\alpha$ that proves the non-existence of optimal shapes in Corollary \ref{c.nonexistence}.}
\label{fig.monotonecirc}
\end{figure}
For every (sufficiently large) $n \in \N$, consider a pair of balls $B^{\alpha_n}_-$ and   $B^{\alpha_n}_+$ included, respectively, in the ball $B_1$ and in the spherical shell $B_{2^n}\setminus B_{2^{n-1}}$, and such that $|B^{\alpha_n}_-|=m(\alpha_n)|B^{\alpha_n}_+|$, 
where $m(\alpha_n)$ is given by Theorem~\ref{t.main} with $\alpha_n$ as parameter (the existence of such a pair of balls is always guaranteed, at least for sufficiently large $n$).
By definition of $\lambda(\alpha)$, as infimum also over $L^\infty_\alpha$, and \eqref{isoperimetric} we obtain
    $$\lambda(\alpha)\leq \inf_{\Omega \in \mathcal O}  |\Omega|^{\frac{2}{d}}\lambda_1^g(\Omega) \le |B^{\alpha_n}_+ \cup B^{\alpha_n}_-|^{\frac2d}\lambda_1^g(B^{\alpha_n}_+ \cup B^{\alpha_n}_-) =\lambda(\alpha_n).$$
Letting $n\to\infty$ above, since Theorem~\ref{t.main2} gives $\lambda(\alpha_n)\to \lambda(\alpha)$, we deduce
\[
\lambda(\alpha)= \inf_{\Omega \in \mathcal O}  |\Omega|^{\frac{2}{d}}\lambda_1^g(\Omega),
\]
that is $\lambda(\alpha)$ is the infimum of the shape optimization problem. Now, assume by contradiction the existence of an optimal shape $\Omega_*\in \mathcal O$ (i.e., an open and bounded set in $\mathbb R^d$) reaching the previous infimum. By boundedness of $\Omega_*$ there exists $n_*\in \N$  such that $\Omega_* \subset B_{2^{n_*}}$ and in particular $g\in L^\infty_{\alpha_{n_*}}$. This combined with the definition and the strict monotonicity of the function $\lambda(\cdot)$, see Theorem~\ref{t.main2}, imply
    \begin{equation}\label{PosOm}
       \lambda(\alpha)= |\Omega_*|^{\frac{2}{d}}\lambda_1^g(\Omega_*) \ge \lambda(\alpha_{n_*})> \lambda(\alpha),
    \end{equation}
    which is a contradiction. An optimal shape does not exist for \eqref{p.finale}.
    \end{proof}

\appendix

\section{Auxiliary results}

We present several standard results for proving the main theorems in this paper. Although we were unable to find the exact formulations in current literature, we provide their statements and proofs here for completeness.

\subsection{\!Pohozaev identity and shape derivative in twisted settings}\label{appPoho}
In this section, we recall two standard results for twisted eigenvalues and twisted eigenfunctions.
Given two disjoint non-empty balls $B_+\subset\{x_1>1\}$, $B_-\subset \{x_1<-1\}$ with $|B_+|\ge |B_-|$ we consider in the whole subsection $$\Omega=B_+\cup B_-,\quad \quad \chi_\alpha= \alpha \chi_{\{x_1>0\}}+\chi_{\{x_1\le0\}}$$ and a twisted eigenfunction $u$ corresponding to $\lambda=\lambda_{1}^{\chi_\alpha}(\Omega)$ which, by Lemma~\ref{l.reg}, solves, in the classical sense, the following system
\begin{equation}\label{system.constant}
\begin{cases}
    -\Delta u= \lambda u + \xi \chi_\alpha & \text{in $\Omega$,} \\
    u=0 & \text{on $\partial \Omega$,}\\
    \int_\Omega u(x)\chi_\alpha(x) dx=0,\\
    \int_{\Omega} u(x)^2 dx=1,
\end{cases}
\end{equation}
with $\xi$ as in \eqref{eq.xi} (the latter condition is the normalization of the eigenfunction). 
The following Pohozaev-type identity for twisted eigenfunctions holds.

\begin{prop}\label{p.pohozaev}
Every solution $u$ of \eqref{system.constant} satisfies
    \begin{equation*}
        2\lambda=\int_{\partial \oo} |\nabla u(x)|^2 (X \cdot \nu)(x) \, d\mathcal{H}^{d-1}(x),
    \end{equation*}
     with $X=(x_1,x_2,\dots, x_d)$. 
\end{prop}
\begin{proof}
    As in \cite[equation (9.1.5)]{r} the following integration-by-parts formula holds
    \begin{equation}\label{eq.poz}
     \int_{\partial \oo}\! | \nabla u|^2 (X \cdot \nu) d\mathcal{H}^{d-1} + (2-d)\!\! \int_\oo\! u \Delta u dx=2 \int_\oo \!(\nabla u\! \cdot \!X) \Delta u dx.
    \end{equation}
  The equation in the system \eqref{system.constant}, combined with the orthogonality and normalization conditions satisfied by $u$, allow to compute the integral in the second term of \eqref{eq.poz} as follows
    \begin{equation}\label{dim.a1}
        \int_\oo u(x) \Delta u(x)  dx=-\lambda \int_\oo u(x)^2  dx - \xi \int_\oo u(x) \chi_\alpha(x)  dx= -\lambda,
    \end{equation}
 and similarly, for the integral in the third term of \eqref{eq.poz}
    \begin{equation*}
        \int_\oo (\nabla u(x) \cdot X) \Delta u(x)  dx
        = -\lambda\int_\oo (\nabla u(x) \cdot X)  u(x)  dx - \xi \int_\oo (\nabla u(x) \cdot X)\chi_\alpha(x)  dx.
    \end{equation*}
    By making use of the boundary condition $u=0$ on $\partial \oo$, and the normalization condition $||u||_{2}=1$ again, the divergence theorem yields
    \begin{equation*}
        -\int_\oo (\nabla u(x) \cdot X)  u(x) dx
        =-\int_\oo \textup{div}\left(\frac{1}{2} u(x)^2 X \right)  dx + \int_\oo  \frac{d}{2} u(x)^2  dx
        = \frac{d}{2},
    \end{equation*}
and moreover, since $\nabla \chi_\alpha=0$ in $\Omega$ and $u$ solves \eqref{system.constant},
    \begin{equation*}
        \int_\oo (\nabla u(x) \cdot X) \chi_\alpha(x)  dx
        =\int_\oo \textup{div}\left( u(x) \chi_\alpha(x) X \right)  dx=0.
    \end{equation*}
From the last two relations we may rewrite the third term in \eqref{eq.poz} as follows
\begin{equation*}
        2\int_\oo (\nabla u(x) \cdot X) \Delta u(x) dx
        = \lambda d,
    \end{equation*}
    which together with \eqref{dim.a1} yields the thesis.
\end{proof}

Now, we derive the expression of the shape derivative of twisted eigenvalues.
Let $T>0$ and consider a family of maps $\mathfrak{F}(t)$ that satisfy
    \begin{equation*}
        \mathfrak{F}: t\in [0,T) \rightarrow W^{1,\infty}(\R^d,\R^d) \textup{ differentiable at } 0 \textup{ with } \mathfrak{F}(0) = I, \, \mathfrak{F}'(0) = W ,
    \end{equation*}
    where $I$ is the identity map from $\R^d$ to $\R^d$ and $W$ is a smooth vector field from $\R^d$ to $\R^d$. We denote by $\oo_t=\mathfrak{F}(t)(\oo)$ and for $t$ sufficiently small we may assume $\mathfrak{F}(t)(B_+)\subset \{x_1>0\}$ and $\mathfrak{F}(t)(B_-)\subset \{x_1<0\}$. We consider $u_t \in H^1_{0,\chi_\alpha}(\oo_t)$ a twisted eigenfunction corresponding to $\lambda_t=\lambda_{1}^{\chi_\alpha}(\oo_t)$ normalized so that
    \begin{equation}\label{eq.normalizations}
        \int_{\oo_t} u_t(x) \chi_\alpha(x)  dx=0
        \quad \text{ and } \quad
        \int_{\oo_t} u_t(x)^2  dx=1.
    \end{equation}
    Let $\xi_t=\xi(u_t,\chi_\alpha)$, defined as in Proposition~\ref{p.min}, and denote by $u'$, $\lambda'$ and $\xi'$ the derivatives at $0$ of the functions that associate $t \mapsto
     u_t$, $t \mapsto \lambda_t$, and $t \mapsto \xi_t$, respectively.
     
\begin{prop}\label{p.shapederivative}
Let $\Omega=B_+\cup B_-$ with $|\Omega|^{2/d}\lambda< |B|^{2/d}\lambda_2(B)$ with $B$ any ball in $\mathbb R^d$ and $u$ be an eigenfunction corresponding to $\lambda$. Then  
   \[
    \lambda'=-\int_{\partial\Omega}|\nabla u(x)|^2 (W \cdot \nu)(x) \, d\mathcal{H}^{d-1}(x),
    \]
where $W$ is a smooth vector field from $\R^d$ to $\R^d$.
\end{prop}
\begin{proof}
By Theorem~\ref{simpleCourant} the eigenvalue $\lambda$ is simple. As it is often the case, the shape derivative of a simple eigenvalue can be obtained through a formal calculation, which can be rigorously justified \emph{a posteriori} as in \cite[Section 5.3]{hp}. For every sufficiently small $t\in [0,T)$ consider the family of maps $\mathfrak{F}(t)$ introduced prior to the statement of the Proposition. In particular, the function $u_t$ solves
\begin{equation*}
\begin{cases}
    -\Delta u_t= \lambda_t u_t + \xi_t \chi_\alpha & \text{in $\Omega_t$,} \\
    u_t=0 & \text{on $\partial \Omega_t$,}\\
\end{cases}
\end{equation*}
thus by differentiating w.r.t. $t$ the equations in the previous system as well as \eqref{eq.normalizations}, and evaluating at $t=0$, we obtain
\begin{equation}\label{system}
\begin{cases}
-\Delta u'=\lambda'u+\lambda u'+\xi'\chi_\alpha \quad  & \text{in $\oo$,}\\
u'=-|\nabla u|\, W \cdot \nu \quad & \text{on $\partial \Omega$,} \\
\int_\Omega u'(x)\chi_\alpha(x)  dx=0, \\
\int_{\Omega} u(x) u'(x) dx=0.
\end{cases}
\end{equation}
By multiplying the first equation in \eqref{system} by $u$ and integrating over $\Omega$, then performing an integration by parts and using the orthogonality of $u$ to $u'$ and $\chi_\alpha$ (see \eqref{system} and \eqref{system.constant}), we obtain
\[
\int_\Omega \nabla u'(x)\nabla u(x)dx=\lambda'\int_\Omega u(x)^2 dx+\lambda \int_\Omega u'(x)u(x)dx+\xi'\int_\Omega u(x)\chi_\alpha(x)dx= \lambda'.
\]
Another integration by parts, applied to the integral in the left-hand side of the previous equality, yields 
\[
\int_\Omega \nabla u'(x)\nabla u(x)dx=\lambda\int_\Omega u'(x)u(x) dx+\xi \int_\Omega u'(x)\chi_\alpha(x)dx-\int_{\partial \Omega} u'\, (\nabla u\cdot \nu) d\mathcal H^{d-1},
\]
which combined with the boundary and orthogonality conditions satisfied by $u'$ in \eqref{system}, gives the thesis.
\end{proof}

\subsection{Some results on Bessel functions}
In this subsection, we recall some facts about Bessel functions of arbitrary order.
For $\sigma\in \mathbb R$ and $m \in \N$ let $J_\sigma$ denote the \emph{Bessel function of the first kind} of order $\sigma$ (it is an analytic function of real variable, defined for positive values) and $j_{\sigma,m}$ its $m$-th zero. If $\sigma>-1$ by \cite[Section 15.22, p. 479]{w} $j_{\sigma, 1}<j_{\sigma+1,1}$ and moreover the Bessel function $J_{\sigma}$ is positive in $(0,j_{\sigma,1})$. By \cite[Section 3.2, p. 45]{w}, for every $x>0$, the following  \emph{recurrence formulae} hold:
    \begin{equation}\label{recurrence.relation}
        \frac{2\sigma}{x}J_\sigma(x)- J_{\sigma-1}(x)=J_{\sigma+1}(x),
    \end{equation}
and 
    \begin{equation}\label{bessel.derivative}
     (x^{\sigma} J_\sigma(x))'= x^{\sigma} J_{\sigma-1}(x), \quad  (x^{-\sigma} J_\sigma(x))'=- x^{-\sigma} J_{\sigma+1}(x). 
    \end{equation}
In particular, given $R,\omega>0$, from the first formula for the derivative in \eqref{bessel.derivative} we derive the following \emph{integral representation} of Bessel functions
    \begin{equation}\label{bessel.integral}
    \int_0^R r^{\sigma} J_{\sigma-1}(\omega r) \, dr= \frac{R^{\sigma}}{\omega}  J_{\sigma}(\omega R).
\end{equation}

\begin{prop}\label{p.phi}
For $\sigma \ge 1/2$, the analytic function $\phi_{\sigma}\colon (0,j_{\sigma-1,1})\cup (j_{\sigma-1,1},j_{\sigma,1})\to \mathbb R$ defined by
\begin{equation*}
    \phi_\sigma(x):=x^{2\sigma} \frac{J_{\sigma+1}(x)}{J_{\sigma-1}(x)}, \quad \text{for every $x \in (0,j_{\sigma-1,1})\cup (j_{\sigma-1,1},j_{\sigma,1})$},
\end{equation*}
is positive in $(0,j_{\sigma-1,1})$ and negative in $(j_{\sigma-1,1}, j_{\sigma,1})$. Moreover, $\phi_\sigma$ is strictly increasing in  $(0,j_{\sigma-1,1})\cup (j_{\sigma-1,1},j_{\sigma,1})$.
\end{prop}
\begin{proof}
 The sign of $\phi_\sigma$ follows from the signs of the Bessel functions $J_{\sigma+1}$ and $J_{\sigma-1}$ in the domain of definition of $\phi_\sigma$ (see Porter's theorem above).
By  \eqref{bessel.derivative} and \eqref{recurrence.relation} we have 
   \begin{equation*}
    \begin{split}
         \phi_\sigma'(x)
        &=\left(\dfrac{x^{\sigma+1} J_{\sigma+1}(x)}{x^{-\sigma+1} J_{\sigma-1}(x)} \right)'
        = \frac{x^{2}J_{\sigma}(x)(J_{\sigma-1}(x)+J_{\sigma+1}(x))}{x^{-2\sigma+2} J_{\sigma-1}(x)^2}\\
        &=x^{2\sigma} \frac{J_{\sigma}(x)}{J_{\sigma-1}(x)} \left( \frac{J_{\sigma-1}(x)+J_{\sigma+1}(x)}
    {J_{\sigma-1}(x)} \right)=2\sigma x^{2\sigma-1} \frac{J_{\sigma}(x)^2}{J_{\sigma-1}(x)^2},
    \end{split}
\end{equation*}
which implies the positivity of $\phi'_\sigma$ on its domain of definition.
\end{proof}

\begin{prop}\label{p.phi-increasing}
For $\sigma \ge 1/2$, the analytic function $\Phi_\sigma\colon (0,j_{\sigma,1})\to \mathbb R$ defined by 
\begin{equation*}
    \Phi_\sigma(x):=x^{2\sigma} \frac{J_{\sigma+1}(x)}{J_{\sigma}(x)}, \quad \text{for every $x \in (0,j_{\sigma,1})$},
\end{equation*}      
is positive, strictly increasing with
\begin{equation}\label{derivative.Phi}
        \Phi_\sigma'(x)
        =x^{2\sigma} -x^{-1}\Phi_\sigma(x) +x^{-2\sigma}\Phi_\sigma(x)^2, \quad \text{for every $x \in (0,j_{\sigma,1})$}.
    \end{equation} 
  \end{prop}
\begin{proof}
By the \textit{Mittag-Leffler expansion}, see \cite[Section 15.41, p. 498, equation (1)]{w}, we have
$$\frac{J_{\sigma+1}(x)}{J_{\sigma}(x)}=2x\sum_{m=1}^{\infty} \dfrac{1}{(j_{\sigma,m})^2-x^2}.$$
So, since it is a pointwise convergent sum of positive strictly increasing functions,  the function $J_{\sigma+1}(x)/ J_{\sigma}(x)$ is positive and strictly increasing for every $x \in (0,j_{\sigma,1})$. Therefore, the function $\Phi_\sigma$ is positive and strictly increasing, as it is the product of two positive, strictly increasing functions. Now, by \eqref{bessel.derivative} then 
      \begin{equation*}
       \Phi_\sigma'(x)
       =\left( \dfrac{x^{\sigma+1}J_{\sigma+1}(x)}{x (x^{-\sigma}J_{\sigma}(x))} \right)'
       =x^{2\sigma}-\frac{x^{2\sigma-1}J_{\sigma+1}(x)}{J_{\sigma}(x)}+\frac{x^{2\sigma}J_{\sigma+1}(x)^2}{J_{\sigma}(x)^2},
   \end{equation*}
and also \eqref{derivative.Phi} follows.
\end{proof}

\begin{prop}\label{p.upsilon}
For $\sigma \ge 1/2$, the analytic function $\Upsilon_{\sigma}\colon (0,j_{\sigma,1})\to \mathbb R$ defined by
        \begin{equation*}
            \Upsilon_{\sigma}(x)
            :=x\left( \frac{J_{\sigma+1}(x)}{J_{\sigma}(x)} + \frac{J_{\sigma}(x)}{J_{\sigma+1}(x)} \right)
            =2(\sigma+1) + x\left( \frac{J_{\sigma+1}(x)}{J_{\sigma}(x)} - \frac{J_{\sigma+2}(x)}{J_{\sigma+1}(x)} \right)
        \end{equation*}
        for every $x \in (0,j_{\sigma,1})$, is strictly increasing.
    \end{prop}
    \begin{proof}
    First notice that the last equality in the statement is obtained appealing to \eqref{recurrence.relation}.   By the \textit{Mittag-Leffler expansion}, \cite[Section 15.41, p. 498, equation (1)]{w}, we have
        \begin{equation*}
        \begin{split}
            \frac{J_{\sigma+1}(x)}{J_{\sigma}(x)} - \frac{J_{\sigma+2}(x)}{J_{\sigma+1}(x)}
            &=2x\sum_{m=1}^{\infty} 
            \left( \dfrac{1}{(j_{\sigma,m})^2-x^2}-\dfrac{1}{(j_{\sigma+1,m})^2-x^2} \right)
            \quad \textup{in } (0,j_{\sigma,1}).
        \end{split}
        \end{equation*}
        So it is sufficient to show that the function
        \begin{equation*}
            x \mapsto \sum_{m=1}^{\infty} \dfrac{(j_{\sigma+1,m})^2-(j_{\sigma,m})^2}{[(j_{\sigma,m})^2-x^2][(j_{\sigma+1,m})^2-x^2]}
        \end{equation*}
        is positive and strictly increasing in $(0,j_{\sigma,1})$. Let $n \in \N$, define the function 
        \begin{equation*}
            \theta_n(x):=\frac{1}{[(j_{\sigma,n})^2-x^2][(j_{\sigma+1,n})^2-x^2]} \quad \textup{in } (0,j_{\sigma,1}) \, ,
        \end{equation*}
        differentiating it we obtain
        \begin{equation*}
            \theta'_n(x)=\frac{2x}{[(j_{\sigma,n})^2-x^2]^2[(j_{\sigma+1,n})^2-x^2]}
            + \frac{2x}{[(j_{\sigma,n})^2-x^2][(j_{\sigma+1,n})^2-x^2]^2} > 0.
        \end{equation*}
        Eventually, the sequence of positive strictly increasing functions 
        \begin{equation*}
            \left\{x \mapsto \sum_{m=1}^{n} [(j_{\sigma+1,m})^2-(j_{\sigma,m})^2] \theta_m(x) \right\}_{n \in \N}
        \end{equation*}
        pointwise converge in $(0,j_{\sigma,1})$ to
        \begin{equation*}
            x \mapsto  \sum_{m=1}^{\infty} [(j_{\sigma+1,m})^2-(j_{\sigma,m})^2] \theta_m(x) \, ,
        \end{equation*}
        therefore this function is positive strictly increasing in $(0,j_{\sigma,1})$.
    \end{proof}
    
    \begin{oss}\label{oss.notazione}
    By \eqref{derivative.Phi} for every $x \in (0,j_{\sigma,1})$ there holds
\[
\Upsilon_{\sigma}(x)
            =x\left( \frac{x^{2\sigma}}{\Phi_\sigma(x)} +x^{-2\sigma}\Phi_\sigma(x)\right)
            = \frac{x\Phi_\sigma'(x)}{\Phi_\sigma(x)}+1.
            \]
\end{oss}

\begin{prop}\label{ODE}
    Let $R>0$ and ${\xi} \in \R$. 
    The solutions 
    \begin{equation*}
        (u, \lambda) \in (C^2(0,R) \cap C^1([0,R])) \times \R^+
    \end{equation*}
    of the eigenvalue problem
    \begin{equation*}
    \begin{cases}
        u''(r) + \frac{d-1}{r} u'(r) + \lambda u(r)={\xi} & \text{for } r \in(0,R)\,, \\
        u(R)=0 \, , \ u'(0)=0\, ,
    \end{cases}
\end{equation*}
are obtained by taking $u$ in the family
\begin{equation*}
    \left\{ u_s(r)=s\left( r^{1-\frac{d}{2}}J_{\frac{d}{2}-1}(\sqrt{\lambda} r) - R^{1-\frac{d}{2}}J_{\frac{d}{2}-1}(\sqrt{\lambda} R) \right) \, , \, s \in \R \right\} \, ,
\end{equation*}
and $\lambda$ one of the countably many positive solutions of the equation
\begin{equation}\label{relazione}
    {\xi}= - s \lambda R^{1-\frac{d}{2}} J_{\frac{d}{2}-1}(\sqrt{\lambda} R).
\end{equation}
\end{prop}
\begin{proof}
    Let $Y_n$ be the \emph{Bessel function of the second kind} of order $n$, see \cite[p.~116]{bow}. By applying \cite[equations (6.80), (6.81) and (6.82)]{bow} (with $\alpha=1-\frac{d}{2}, \gamma=1, n=\frac{d}{2} -1$ and $\beta^2=\lambda$) we deduce that the general solution of the differential equation
    \begin{equation*}
        u''(r) + \frac{2n+1}{r} u'(r) + \lambda u(r)=\xi \, ,
    \end{equation*}
    is given by
    \begin{equation}\label{general.solution}
      u(r)=\begin{cases}
      r^{-n} (a J_n(\sqrt{\lambda} r) + b Y_n(\sqrt{\lambda} r)) + {\xi}/\lambda & \text{if $n \in \Z,$} \\
        r^{-n} (a J_n(\sqrt{\lambda} r) + b J_{-n}(\sqrt{\lambda} r)) + {\xi}/\lambda & \text{if  $n \notin \Z$,}
    \end{cases}
    \end{equation}
for every $a,b \in \R$. 
Since, by \cite[below equation (6.74) p.~116]{bow} and \cite[equations (6.2) and (6.3) p.~87]{bow} we have
    \begin{equation*}
         \lim_{r \rightarrow 0^+} \bigg|\left( r^{-n} Y_n(\sqrt{\lambda} r) \right)'\bigg|
        =+ \infty
        \quad , \quad
        \lim_{r \rightarrow 0^+} \left( r^{-n} J_{n}(\sqrt{\lambda} r) \right)'=0 \quad \text{if $n \in \Z,$}
    \end{equation*}
    and by \cite[equation (6.5) p. 88]{bow} we have
    \begin{equation*}
        \,\,\,\,\,\,\,\,\,\,\,\lim_{r \rightarrow 0^+} \bigg|\left( r^{-n} J_{-n}(\sqrt{\lambda} r) \right)' \bigg| 
        =+ \infty  \quad , \quad
        \lim_{r \rightarrow 0^+} \left( r^{-n} J_{n}(\sqrt{\lambda} r) \right)'=0 
        \quad  \text{if $n \in  \mathbb R^+ \setminus  \mathbb N,$} 
    \end{equation*}
    then from $u'(0)=0$ we obtain $b=0$ in \eqref{general.solution}.
    Thus we can restrict to the class of solutions of the form
    \begin{equation*}
      u(r)= ar^{-n}  J_n(\sqrt{\lambda} r)+ {\xi}/\lambda 
    \end{equation*}
    for $a \in \R$. The condition $u(R)=0$ is equivalent to find the $\sqrt{\lambda}$ that solve 
\begin{equation}\label{homogenous.term}
    \lambda  J_n(\sqrt{\lambda} R)=-\frac{{\xi}}{a}R^n .
\end{equation}
Consider the function $f \colon \R^+ \rightarrow \R$ defined by $f(\lambda):=\lambda J_n(\sqrt{\lambda} R)$ and since by truncating \emph{Hankel’s expansion}\footnote{This particular formula was first shown by Jacobi, see \cite[Section 7.1, p. 195]{w}.}, see \cite[Section 7.1, p. 195]{w}, we have
    \begin{equation*}
        J_n (x) \sim \sqrt{\frac{2}{\pi x}} \cos \Big( x - \frac{2 n +1}{4} \pi \Big) + O \Big( \frac{1}{x^{3/2}} \Big) \quad \text{as $x \rightarrow \infty$} \, ,
    \end{equation*}
there exist two strictly increasing sequences of positive numbers $\lambda_n^+$ and $\lambda_n^-$ such that
\begin{equation*}
    \lim_{n \rightarrow +\infty} f(\lambda_n^+)=+\infty
    \quad , \quad
    \lim_{n \rightarrow +\infty} f(\lambda_n^-)=-\infty.
\end{equation*}
Since $f$ is an analytic function of real variable we have $\textup{Im}(f)=\R$, moreover $f^{-1}(-{\xi}R^n/a)$ has countably many elements and \eqref{homogenous.term} can be rewritten as 
          ${{\xi}}=-a{\lambda}R^{-n}J_n(\sqrt{\lambda} R)$,
that is the last relation.
\end{proof}

\medskip

\noindent{\textbf{Acknowledgements.}}
The authors are grateful to Mark S. Ashbaugh, for sharing the valuable historical background  regarding the Hong-Krahn-Szego inequality, to Fabio Nicola and Diego Noja, for providing  physical evidence supporting the twisted eigenvalues (in Quantum Mechanics and in stability theory of dispersive equations, respectively), and to Pierre Bérard, for providing feedback on the results related to nodal domains of twisted eigenfunctions. The authors are also very grateful to Dorin Bucur for stimulating discussions and for suggesting an approch to proving the nodal domain theorem in Theorem~\ref{simpleCourant}. Two anonymous referees are also acknowledged for important comments on a preliminary version of the paper. The figures in the paper have been obtained using the software MATLAB.
This publication is part of the project PNRR-NGEU which has received funding from the MUR – DM 351/2022. Both the authors are members of the Research Group INdAM–GNAMPA. The second author has been partially supported by the PRIN 2022 project 2022R537CS \emph{NO3
- Nodal Optimization, NOnlinear elliptic equations, NOnlocal geometric problems, with a focus on regularity}, founded by the European Union - Next Generation EU.

\end{document}